\documentclass[preprint, authoryear, 12pt]{elsarticle}
\usepackage{mathrsfs}
\usepackage{amsmath}
\usepackage{amsfonts}
\usepackage{amssymb}
\usepackage{amsthm}

\newcommand{\ee}{\mathbb E}
\newcommand{\rr}{\mathbb R}

\newcommand{\zz}{\mathbb Z}
\newcommand{\nn}{\mathbb N}

\DeclareMathOperator*{\esssup}{ess\,sup}

\makeatletter

\newcommand{\Rmnum}[1]{\expandafter\@slowromancap\romannumeral #1@}
\renewcommand{\section}{\@startsection
{section}
{1}
{0mm}
{-\baselineskip}
{0.5\baselineskip}
{\normalfont\large\bfseries}} 
\makeatother

\theoremstyle{plain}
\newtheorem{theorem}{Theorem}[section]
\newtheorem{proposition}[theorem]{Proposition}
\newtheorem{corollary}[theorem]{Corollary}
\newtheorem{lemma}[theorem]{Lemma}

\theoremstyle{definition}

\numberwithin{equation}{section}

\allowdisplaybreaks[1]
\begin{document}

\begin{frontmatter}

\date{}

\title{On the
integrated squared error of the linear wavelet density estimator}

\author{Lu Lu\corref{cor}}
\ead{lu.lu@uconn.edu}

\cortext[cor]{Principal Corresponding Author}
\address{Department of Mathematics, University of Connecticut, Storrs, CT, USA, 06269}

\begin{abstract}
Linear wavelet density estimators are wavelet projections of the
empirical measure based on independent, identically distributed
observations. We study here the law of the iterated logarithm (LIL)
and a Berry-Esseen type theorem. These results are proved under
different assumptions on the density $f$ that are different from
those needed for similar results in the case of convolution kernels
(KDE): whereas the smoothness requirements are much less stringent
than for the KDE, Riemann integrability assumptions are needed in
order to compute the asymptotic variance, which gives the scaling
constant in LIL. To study the Berry-Esseen type theorem, a rate of
convergence result in the martingale CLT is used.
\end{abstract}


\end{frontmatter}

\section{Introduction}

Let $X,X_1,X_2,...$be i.i.d random variables in $\mathbb{R}$ with
common Lebesgue density $f$. Let $\phi\in L_2(\rr)$ be a scaling
function and $\psi$ the corresponding wavelet function. Let
$\phi_{0k}:=\phi(x-k)$ and $\psi_{jk}:=2^{j/2}\psi(2^jx-k)$.
$\{\phi_{0k},\psi_{jk}\}$ forms an orthonormal system in $L_2(\rr)$.
Every $f\in L_p(\rr)$ has a formal expansion
\begin{equation}
f(x)=\sum_k
\alpha_{0k}\phi_{0k}(x)+\sum_{j=0}^\infty\sum_k\beta_{jk}\psi_{jk}(x).
\end{equation}
The linear wavelet density estimator is defined as
\begin{equation}        \label{eq: f_n}
\hat{f}_{n}(x)=\sum_k
\hat{\alpha}_{0k}\phi_{0k}(x)+\sum_{j=0}^{j_n-1}\sum_k\hat{\beta}_{jk}\psi_{jk}(x),
\end{equation}
where $j_n$ is a sequence of integers. $\hat{\alpha}_{jk}$ and
$\hat{\beta}_{jk}$ are constructed by the plug-in method. Let
$P_n=\frac{1}{n}\sum_{i=1}^n\delta_{X_i}$ be the empirical measure
corresponding to the sample $\{X_i\}_{i=1}^n, n\in\nn$. Then
\begin{equation}
\hat{\alpha}_{jk}=P_n(\phi_{jk})=\frac{1}{n}\sum_{i=1}^n2^{j/2}\phi(2^jX_i-k),
\end{equation}
\begin{equation}
\hat{\beta}_{jk}=P_n(\psi_{jk})=\frac{1}{n}\sum_{i=1}^n2^{j/2}\psi(2^jX_i-k).
\end{equation}
They are unbiased estimators of $\alpha$ and $\beta$.

The use of this estimator first appeared in Doukhan and Le\'{o}n
(1990) and Kerkyacharian and Picard (1992). When $\phi$ satisfies
certain properties, i.e., bounded and compactly supported, one may
write $\hat{f}_{n}(x)$ in a form similar to that of the classical
kernel density estimator:
\begin{equation}
f_{n,K}(x):=\hat{f}_{n}(x)=\frac{2^{j_n}}{n}\sum_{i=1}^n
K(2^{j_n}x,2^{j_n}X_i),
\end{equation}
where the projection kernel $K(x,y)$ is given by
\begin{equation} \label{eq: K} K(x,y)=\sum_{k\in
\zz}\phi(x-k)\phi(y-k).
\end{equation}
$\{2^{-j_n}\}$ is playing the role of the bandwidth in the classical
kernel density estimation, and the sum is finite for each $x$ and
$y$. By Lemma 8.6, H\"{a}rdle, Kerkyacharian, Picard and Tsybakov
(HKPT, 1998), $K(x,y)$ is majorized by a convolution kernel
$\Phi(x-y)$ such that
\begin{equation} \label{eq: HKPT}
|K(x,y)|\leq\Phi(x-y),
\end{equation}
where $\Phi:\rr\rightarrow\rr^+$ is a bounded, compactly supported
and symmetric function.

A widely accepted measure of performance of an estimator is its mean
integrated squared error, which is the expected value of the
integrated squared error or $L_2$ error defined by  $I_n:=\int
(f_{n}(x)-f(x))^2dx$ (see, e.g., Bowman 1985). The integrated
squared error $I_n$ constitutes in itself a nice global measure of
approximation of the density. And it is of interest to obtain the
asymptotically exact almost sure rate of approximation, in this
measure, of the density by an estimator of interest, often a law of
the iterated logarithm. This was done by Gin\'{e} and Mason (2004)
for kernel density estimators, and it is done here for wavelet
density estimators. We will refer to several results by Gin\'{e} and
Mason (2004), which will be abbreviated as (GM) in what follows.
This type of theorems may be thought of as companion results to
central limit theorems: whereas the latter gives rate of
approximation in probability, the former deals with a.s. rate of
convergence. The central limit theorem for the integrated squared
error $I_n$ was obtained by Hall (1984) for kernel density
estimators, and by Zhang and Zheng (1999) for wavelet density
estimators. We also prove a Berry-Esseen type theorem as a
complement to Zhang and Zheng's result. Doukhan and Le\'{o}n (1993)
obtained a bound on the rate of convergence in the CLT for
generalized density projection estimates with respect to Prohorov's
metric. However, their bound does not apply to the optimal window
width.

To study the integrated square error for the wavelet density estimator, we shall impose the following conditions:\\
\smallskip
\noindent $(f)$: $f(x)$ is bounded.\\
(S1): The scaling function $\phi$ is bounded and compactly
supported (e.g., Daubechies wavelet).\\
\smallskip
Then, in $\eqref{eq: HKPT}$, we can assume $\Phi$ is supported on
$[-A,A]$ for some $A>0$. Set $\theta_\phi(x)=\sum_k|\phi(x-k)|$.
(S1) also guarantees that (see section 8.5, HKPT, 1998),
\begin{equation} \label{eq: phi}
\esssup_x\theta_\phi(x)<\infty.
\end{equation}
\smallskip
\noindent (S2): $\|\phi\|_v<\infty$, where $\|\cdot\|_v$ denotes the
total
variation norm of $\phi$.\\
The bandwidth $\{2^{-j_n}\}$ satisfies\\
\smallskip
\noindent(B1): $j_n\rightarrow \infty,\ 2^{-j_n}\asymp n^{-\delta} \
\ \ \rm{for \ some}\ \delta\in(0,1/3),$ where $a_n\asymp b_n$ means
$0<\liminf a_n/b_n<\limsup
a_n/b_n<\infty$. \\
\smallskip
\noindent(B2): There exists an increasing sequence of positive
constants $\{\lambda_k\}_{k\geq1}$ satisfying
\begin{equation}
\lambda_{k+1}/\lambda_k\to 1,\ \log\log\lambda_k/\log k\to 1, \
\lambda_{k+1}-\lambda_k\to\infty
\end{equation}
as $k\to \infty$, such that $2^{-j_n}$ is constant for
$n\in[\lambda_k,\lambda_{k+1}), \ k\in\nn$. For instance, the
sequence $\lambda_k=\exp(k/\log(e+k))$ satisfies these conditions.

We will prove the following theorems for the statistic
\begin{equation}
J_n:=\|f_{n,K}-f\|_2^2-\ee\|f_{n,K}-f\|_2^2.
\end{equation}
\begin{theorem}  \label{Theorem 5.1}
Let $f, \phi$ and ${j_n}$ satisfy hypotheses (f), (S1), (S2), (B1)
and (B2). Set $\sigma^2:=2\int_\rr f^2(x)dx$. Then,
\begin{equation}        \label{eq: Jn}
\limsup_{n\to\infty}\pm\frac{n2^{-j_n/2}}{\sigma\sqrt{2\log\log
n}}J_n=1, \ \ \ \ a.s.
\end{equation}
\end{theorem}

\begin{theorem}  \label{Theorem BE}
Assume the hypotheses (f), (S1), (B1) and that there exists $L\geq0$
such that $f$ is H\"{o}lder continuous with exponent $0<\alpha\leq1$
on $[-L,L]$: f is monotonically increasing on $(-\infty,-L]$ and
monotonically decreasing on $[L,\infty)$. Let
$Z\sim\mathcal{N}(0,1)$. Then there exists a constant $C$ (depending
on $f$, $\phi$ and $\{j_n\}$), such that
\begin{equation}      \label{eq: BE thm}
\sup_t|\Pr\{n2^{-j_n/2}J_n\leq t\}-\Pr\{\sigma Z\leq t\}|\leq
C(n^{-3\delta/16}\vee n^{-\alpha\delta}\sqrt{\log n})
\end{equation}
where $\sigma^2=2\int_\rr f^2(x)dx$.
\end{theorem}
For example, if $2^{-j_n}\asymp n^{-1/5}$,
$\sup_t|\Pr\{n2^{-j_n/2}J_n\leq t\}-\Pr\{\sigma Z\leq t\}|\leq
C(n^{-3/80}\vee n^{-\alpha/5}\sqrt{\log n})$. No claim of optimality
of the rate obtained is made.

Zhang and Zheng (1999) used the fact that $J_n$ coincides with its
stochastic part, $\bar{J}_n$, where
\begin{equation} \label{eq: defJnb}
\bar{J}_n:=\|f_{n,K}-\ee f_{n,K}\|_2^2-\ee\|f_{n,K}-\ee
f_{n,K}\|_2^2.
\end{equation}
This is due to the orthogonality of the wavelet basis. We will
include a short proof later for completeness. Thus, there is no need
to analyze the bias part and assume more regularity conditions on
the density $f$ as is done in the kernel case (e.g., Hall, 1984; GM,
2004).

Next we set up some notations. Let $K$ be the projection kernel
associated with the scaling function $\phi$ as in $\eqref{eq: K}$.
Set
$$K_n(t,x):=K(2^{j_n}t,2^{j_n}x)\ \ {\rm and}\
\bar{K}_n(t,x):=K_n(t,x)-\ee K_n(t,X).$$ Then by $\eqref{eq:
defJnb}$,
\begin{equation}
\begin{split}   \label{eq: Jnbar}
\bar{J}_n&=\frac{2^{2j_n}}{n^2}\left[\int_\rr\left(\sum_{i=1}^n\bar{K}(2^{j_n}t,2^{j_n}X_i)\right)^2dt
    -\ee \int_\rr\left(\sum_{i=1}^n\bar{K}(2^{j_n}t,2^{j_n}X_i)\right)^2dt\right] \\
    &=:\frac{2^{2j_n}}{n^2}W_n(\rr),
\end{split}
\end{equation}
where
\begin{equation}
\begin{split}  \label{eq: Wn dec}
W_n(F):&=\int_F\left(\sum_{i=1}^n\bar{K}(2^{j_n}t,2^{j_n}X_i)\right)^2dt
    -\ee\int_F\left(\sum_{i=1}^n\bar{K}(2^{j_n}t,2^{j_n}X_i)\right)^2dt \\
      &= U_n(F)+L_n(F),
\end{split}
\end{equation}
\begin{equation}
U_n(F)=\sum_{1\leq i\neq j\leq n}\int_F
\bar{K}_n(t,X_i)\bar{K}_n(t,X_j)dt, \ \ L_n(F)=\sum_{i=1}^n\int_F
\left(\bar{K}_n^2(t,X_i)-\ee\bar{K}_n^2(t,X)\right)dt.
\end{equation}
The measurable set $F$ will normally be $\rr$, $[-M,M]$ or
$[-M,M]^C$, $M>0$, with $\int_F f(t)dt>0$. But in the results below,
$F$ can be any set with this property and such that
\begin{equation} \label{eq: F}
\lambda(\{x+y:x\in F,|y|<\varepsilon\}\cap F^c)\to 0 \ {\rm as}\
\varepsilon\to 0.
\end{equation}

The proof of Theorem $\ref{Theorem 5.1}$ for the most part follows
the same pattern in (GM): For some $M$ large enough, $W_n([-M,M]^C)$
is shown to be negligible by using an exponential inequality for
degenerate U-statistics (Gin\'{e}, Lata{\l}a and Zinn, 2000) and
Bernstein's inequality for the diagonal term. Therefore, we may
truncate $\bar{J}_n$ and deal with $W_n([-M,M])$. This is
approximated by a Gaussian chaos using strong approximations
(Koml\'{o}s-Major-Tusn\'{a}dy inequality) and a moderate deviation
is proved for it. Finally, one deals with the usual blocking of laws
of the iterated logarithm. Here it can be implemented again because
of Bernstein type exponential inequalities for $U$-statistics.
However, due to the fact that $K(x,y)$ is not a convolution kernel,
the computation of the limiting variance turns out to be a major
difficulty, which we surmount using ideas from the proof of CLT in
Zhang and Zheng (1999). For this we require $f$ to be (improper)
Riemann integrable on $\rr$, and this is the purpose of condition
$(f)$.

In order to get the convergence rate in CLT, we need to assume more
conditions on $f$. $\bar{J}_n$ is composed of $L_n(\rr)$ and
$U_n(\rr)$. The exponential inequality for U-statistics is used to
show $L_n(\rr)$ is negligible. Then $U_n(\rr)$ is approximated by a
martingale and the rate of convergence was obtained using Erickson,
Quine and Weber (1979)'s result. The U-statistics method and the
application of the martingale limit theory can be traced back to
Hall (1984). It makes the study of $L_2$ error easier, but it does
not apply to $L_p$ error if $p\neq2$.

The article is organized as follows. In section 2 we collect the
variance computation results. In section 3 we state results of tail
estimation. In section 4, we obtain a moderate deviation result for
$W_n([-M,M])$. In section 5, we complete the proofs of Theorem
$\ref{Theorem 5.1}$ and Theorem $\ref{Theorem BE}$. In the appendix,
we give proofs to some lemmas stated in section 2. $C$ is a
universal constant which might differ from line to line.

\section{Variance Computations}

  We present here some inequalities and variance computations used
throughout the paper. Only the exact limits present problems and
must be treated differently than in the case of convolution kernels,
but upper bounds can be dealt with essentially as in the convolution
kernel case because of the majorization property $\eqref{eq: HKPT}$.
We will state these results without giving detailed proofs. They can
be verified by replacing, in the corresponding proofs by (GM), the
bandwidth $h_n$ by $2^{-j_n}$ and the projection kernel $K(x,y)$ by
a convolution kernel $\Phi(x-y)$ that is given by $\eqref{eq:
HKPT}$. More specifically, if the kernel $K(x,y)$ satisfies
$\eqref{eq: HKPT}$, we have the following estimates: For all $x$ and
$y$, and all measurable sets $F$,
\begin{equation} \label{eq: 2.18}
\int_F \bar{K}_n^2(t,x)dt\leq 4\cdot2^{-j_n}\|\Phi\|_2^2,
\end{equation}
\begin{equation} \label{eq: 2.17}
\left|\int_F \bar{K}_n^2(t,x)dt-\ee\int_F
\bar{K}_n^2(t,X)dt\right|\leq 8\cdot2^{-j_n}\|\Phi\|_2^2,
\end{equation}
and by Cauchy-Schwarz,
\begin{equation}  \label{eq: 2.16}
\int_F \left|\bar{K}_n(t,x)\bar{K}_n(t,y)\right|dt\leq
4\cdot2^{-j_n}\|\Phi\|_2^2.
\end{equation}
We have an analogue to Corollary 2.7, (GM).
\begin{corollary}\label{var}
Assume (f), (S1) and (B1) hold, and that F satisfies condition
$\eqref{eq: F}$. Then there exists $n_0=n_0(F)$ such that, for all
$n\geq n_0$,
\begin{equation}
{\rm Var}\int_F\bar{K}_n^2(t,X)dt\leq8\cdot
2^{-2j_n}\|\Phi\|_2^4\int_F f(x)dx.
\end{equation}
And for all $n$,
\begin{equation} \label{eq: var bd}
{\rm Var}\int_F\bar{K}_n^2(t,X)dt\leq4\cdot 2^{-2j_n}\|\Phi\|_2^4.
\end{equation}
\end{corollary}
Set
\begin{equation} \label{eq: Rn}
C_n(t,s): =2^{j_n}\int_\rr K_n(t,x)K_n(s,x)f(x)dx,
\ \ R_n(t,s): =2^{j_n}\int_\rr\bar{K}_n(t,x)\bar{K}_n(s,x)f(x)dx.
\end{equation}
Define the operator $\mathcal{R}_{n,F}$ for $\varphi\in L_2(F)$,
\begin{equation} \label{eq: operator}
\mathcal{R}_{n,F}\varphi(s)=\int_FR_n(s,t)\varphi(t)dt.
\end{equation}
The next three lemmas are similar to Lemmas 2.3, 2.4 and 2.5, (GM).
\begin{lemma} \label{Lemma 2.3}
Under the hypotheses of Corollary $\ref{var}$, for the operator
$\mathcal{R}_{n,F}$, we have
\begin{equation}
\sup\{\|\mathcal{R}_{n,F}\varphi\|_2^2: \|\varphi\|_2=1,\varphi\in
L_2(F)\}\leq2^{-2j_n}C(\Phi,f),
\end{equation}
where
\begin{equation} \label{eq: CPhi}
C(\Phi,f)=2\|\Phi\|_1^4\left(\|f\|_\infty^2+\|f\|_2^4\right).
\end{equation}
\end{lemma}
\begin{lemma} \label{Lemma 2.4b}
Under the hypotheses of Corollary $\ref{var}$,
\begin{equation}
\begin{split}    \label{eq: lemma 2.4}
\limsup_{n\to\infty}2^{j_n}\int_{F^2}C_n^2(s,t)dsdt&\leq\int_F
f^2(x)dx\int_\rr\left(\int_\rr\Phi(w+u)\Phi(w)dw\right)^2du\\
&\leq\int_Ff^2(x)dx\|\Phi\|_1^2\|\Phi\|_2^2.
\end{split}
\end{equation}
\end{lemma}
\begin{lemma} \label{Lemma 2.5}
Under the hypotheses of Corollary $\ref{var}$,
\begin{equation} \label{eq: lemma 2.5}
2^{j_n}\int_{F^2}(C_n(s,t)-R_n(s,t))^2dsdt\leq
2^{-j_n}\|\Phi\|_1^4\|f\|_2^4\to0 \ \ {\textrm as}\ \ n\to\infty.
\end{equation}
\end{lemma}

Note that in Lemma $\ref{Lemma 2.4b}$, we can only get an upper
bound instead of the limit using the same method from the
convolution kernel case. Calculation of the exact limit of
$2^{j_n}\int_{[-M,M]^2}R_n^2(s,t)dsdt$ is the key to obtaining the
scaling constant in LIL. By Lemma $\ref{Lemma 2.5}$, we shall
approximate it by $2^{j_n}\int_{[-M,M]^2}C_n^2(s,t)dsdt$ and
calculate the limit of this quantity.
\begin{lemma} \label{Lemma 2.4}
Assume (f) and (B1) holds, and the scaling function $\phi$ satisfies
(S1) such that the kernel K associated with $\phi$ is dominated by
$\Phi$ whose support is contained in $[-A,A]$, where $A$ is an
integer. Then for any $M>0$,
\begin{equation}
\label{eq: limit}
\lim_{n\to\infty}2^{j_n}\int_{[-M,M]^2}C_n^2(s,t)dsdt
=\int_{-M}^Mf^2(y)dy.
\end{equation}
\end{lemma}

In order to prove Theorem $\ref{Theorem BE}$, we need to estimate
how fast $2^{j_n}\int_{\rr^2}C_n^2(s,t)dsdt$ converges to $\int_\rr
f^2(y)dy$. This can be done by imposing more regularity conditions
on $f$.
\begin{lemma}  \label{Lemma Cn2}
Under the hypotheses of Lemma $\ref{Lemma 2.4}$, and assume that, in
addition, $f$ is H\"{o}lder continuous with exponent $0<\alpha\leq1$
on [-L,L], and monotone on tails $(-\infty,-L]\cup[L,\infty)$, where
$L\geq0$. Then for all $n$, there exists a constant $C$ (depending
on $f$, $\phi$ and $\{j_n\}$), such that
\begin{equation} \label{eq: Cn2 bd}
\left|2^{j_n}\int_{\rr^2}C_n^2(s,t)dsdt-\int_\rr f^2(y)dy\right|\leq
Cn^{-\delta\alpha}
\end{equation}
where $\delta\in(0,1/3)$ is the same as in (B1).
\end{lemma}

Together with Lemma $\ref{Lemma 2.5}$, we obtain
\begin{corollary}   \label{Lemma Rn2}
Assume the same conditions in Lemma $\ref{Lemma Cn2}$, for all $n$
sufficiently large depending on $f$ and $\{j_n\}$,
\begin{equation}
\left|2^{j_n}\int_{\rr^2}R_n^2(s,t)dsdt-\int_\rr f^2(y)dy\right|\leq
C(n^{-\delta/2}+n^{-\delta\alpha}),
\end{equation}
where the constant $C$ depends on $f$, $\phi$ and $\{j_n\}$.
\end{corollary}
The proofs of Lemmas $\ref{Lemma 2.4}$ and $\ref{Lemma Cn2}$ are
provided in the appendix.

\section{Tail Estimation}

  The goal of this section is to obtain exponential inequalities for
$W_n(F)$, where F satisfies $\eqref{eq: F}$ and also for
$W_n(\rr)-W_{n,m}(\rr)$. We assume throughout this section that
$\phi$ satisfies (S1), and $K$ is associated with $\phi$ given by
$\eqref{eq: K}$.

Set, for $m<n$,
\begin{equation} \label{eq: Wnm}
W_{n,m}(\rr):=\int_\rr\left[\left(\sum_{m<i\leq
n}\bar{K}_n(t,X_i)\right)^2-\ee\left(\sum_{m<i\leq
n}\bar{K}_n(t,X_i)\right)^2\right]dt
\end{equation}
and \begin{equation} \label{eq: Hn} H_n(x,y):=\int_\rr
\bar{K}_n(t,x)\bar{K}_n(t,y)dt, \ \ H_{n,F}(x,y)=\int_F
\bar{K}_n(t,x)\bar{K}_n(t,y)dt.
\end{equation}
With this notation,
\begin{equation}  \label{eq: 3.3}
U_n(F)=\sum_{1\leq i\neq j\leq n}H_{n,F}(X_i,X_j),\
L_n(F)=\sum_{i=1}^n\left(H_{n,F}(X_i,X_i)-\ee
H_{n,F}(X_i,X_i)\right),
\end{equation}
and
\begin{equation}
\begin{split}   \label{eq: 3.4}
&\quad W_n(\rr)-W_{n,m}(\rr)\\
&=2\sum_{i=1}^m\sum_{j=m+1}^n H_n(X_i,X_j)+\sum_{1\leq i\neq j\leq
m}H_n(X_i,X_j)+\sum_{i=1}^m \left(H_n(X_i,X_i)-\ee
H_n(X_i,X_i)\right).
\end{split}
\end{equation}
Bernstein's inequality (e.g., de la Pe\~{n}a and Gin\'{e}, 1999)
says that for centered, i.i.d. random variables $\xi_i$, if
$\|\xi_i\|_\infty\leq c<\infty$ and $\sigma^2=E\xi_i^2$, then
\begin{equation}
{\rm
Pr}\left\{\sum_{i=1}^m\xi_i>t\right\}\leq\exp\left(-\frac{t^2}{2m\sigma^2+2ct/3}\right).
\end{equation}
Applying it to the 3rd term in the above equation, given Corollary
$\ref{var}$, and inequality $\eqref{eq: 2.17}$, we obtain
\begin{equation}
\begin{split}  \label{eq: 3.5}
&\quad{\rm Pr}\left\{\left|\sum_{i=1}^m
\left(H_n(X_i,X_i)-\ee H_n(X_i,X_i)\right)\right|>\tau n2^{-{3\over 2}j_n}\right\}\\
&\leq2\exp\left(-\frac{\tau^2n^22^{-3j_n}}{8m2^{-2j_n}\|\Phi\|_2^4+{16\over
3}\tau n2^{-{5\over 2}j_n}\|\Phi\|_2^2}\right).
\end{split}
\end{equation}
The first two terms in $\eqref{eq: 3.4}$ are of U-statistics type.
They can be controlled by the following exponential inequality for
canonical U-statistics.
\begin{theorem}(Gin\'{e}, Lata{\l}a, Zinn, 2000)  \label{Theorem 3.1}
There exists a universal constant $L<\infty$ such that, if $h_{i,j}$
are bounded canonical kernels of two variables for the independent
random variables $(X_i^{(1)}$, $X_j^{(2)})$, $i,j=1,2,...,n$, and if
$A,B,C,D$ are as defined below, then
\begin{equation}    \label{eq: 3.6}
{\rm Pr}\left\{\left|\sum_{1\leq i,j\leq
n}h_{i,j}(X_i^{(1)},X_j^{(2)})\right|\geq x\right\}\leq
L\exp\left[-\frac{1}{L}\min\left(\frac{x^2}{C^2},\frac{x}{D},
\frac{x^{2/3}}{B^{2/3}},\frac{x^{1/2}}{A^{1/2}}\right)\right]
\end{equation}
for all $x>0$, where
\begin{equation}
\begin{split}   \label{eq: D}
D&=\|(h_{i,j})\|_{L^2\to L^2} \\
&:=\sup\left\{\ee\sum_{i,j}h_{i,j}(X_i^{(1)},X_j^{(2)})f_i(X_i^{(1)})g_j(X_j^{(2)}):
\ee\sum_if_i^2(X_i^{(1)})\leq1,\ee\sum_jg_j^2(X_j^{(2)})\leq1\right\},
\end{split}
\end{equation}
\begin{equation}
C^2=\sum_{i,j}\ee h_{i,j}^2(X_i,X_j),
\end{equation}
\begin{equation}
B^2=\max_{i,j}\left[\left\|\sum_i\ee
h_{i,j}^2(X_i^{(1)},y)\right\|_\infty, \left\|\sum_j\ee
h_{i,j}^2(x,X_j^{(2)})\right\|_\infty\right]
\end{equation}
and
\begin{equation}
A=\max_{i,j}\|h_{i,j}\|_\infty.
\end{equation}
\end{theorem}
Theorem $\ref{Theorem 3.1}$ also holds if the decoupled U-statistic
$\sum_{1\leq i,j\leq n}h_{i,j}(X_i^{(1)},X_j^{(2)})$ is replaced by
the undecoupled U-statistic $\sum_{1\leq i\neq j\leq
n}h_{i,j}(X_i,X_j)$. We will take $h_{i,j}=H_{n,F,i,j}=H_{n,F}$,
calculate the constants $A,B,C,D$ in Theorem $\ref{Theorem 3.1}$,
and apply it to $\sum_{1\leq i\neq j\leq m}H_{n,F}(X_i,X_j)$.
$\eqref{eq: 2.16}$ gives
\begin{equation} \label{eq: AB bound}
A\leq 4\cdot2^{-j_n}\|\Phi\|_2^2, \ \ B^2\leq
16m\cdot2^{-2j_n}\|\Phi\|_2^4.
\end{equation}
By Lemmas $\ref{Lemma 2.4b}$ and $\ref{Lemma 2.5}$, for $n$ large
enough depending on $F$,
\begin{equation}   \label{eq: C bound}
C^2\leq2m^2\cdot2^{-3j_n}\|\Phi\|_1^2\|\Phi\|_2^2\int_F f^2(x)dx.
\end{equation}
If $f$ satisfies condition $(f)$ and $\phi$ satisfies condition
(S1), the bound on $D$ can be calculated  by following the proof in
the kernel case and making obvious modifications there.
\begin{equation}   \label{eq: D bound}
D\leq4m2^{-2j_n}\|f\|_\infty\|\Phi\|_1^2.
\end{equation}

\begin{proposition}   \label{prop 3.4}
Let $X_i$ be i.i.d. with density f satisfying condition (f). Let F
be a measurable subset of \ $\rr$ satisfying condition $\eqref{eq:
F}$. $\phi$ satisfies $(S1)$ and K is the projection kernel
associated with $\phi$. $2^{-j_n}\to 0$. Then there exist constants
$\kappa_0$ (depending on f and $\phi$) and $n_0$(depending on F, f,
$\phi$ and the sequence $\{j_n\}$) such that, for all $\tau>0$ and
for all $n\geq n_0$, $0\leq m<n$,
\begin{equation}
\begin{split} \label{eq: 3.12}
&\quad{\rm Pr}\left\{\left|\sum_{1\leq i\neq j\leq m}H_{n,F}(X_i,X_j)\right|\geq\tau n2^{-{3\over 2}j_n}\right\}\\
&\leq\kappa_0\exp\left(-{1\over\kappa_0}\min\left[\frac{\tau^2n^2}{m^2\int_Ff^2(x)dx},\frac{\tau
n}{m2^{-j_n/2}},\frac{\tau^{2/3}n^{2/3}2^{-j_n/
3}}{m^{1/3}},\tau^{1/2}n^{1/2}2^{-j_n/4}\right]\right)
\end{split}
\end{equation}
and
\begin{equation}
\begin{split} \label{eq: 3.13}
&\quad{\rm Pr}\left\{\left|\sum_{i=1}^m\sum_{j=m+1}^n H_{n,F}(X_i,X_j)\right|\geq\tau n2^{-{3\over 2}j_n}\right\}\\
&\leq\kappa_0\exp\left(-{1\over\kappa_0}\min\left[\frac{\tau^2n^2}{m(n-m)\int_Ff^2(x)dx},
\frac{\tau n}{\sqrt{m(n-m)}2^{-j_n/2}},\right.\right.\\
&\quad\left.\left.\frac{\tau^{2/3}n^{2/3}2^{-j_n/
3}}{(m\vee(n-m))^{1/3}},\tau^{1/2}n^{1/2}2^{-j_n/ 4}\right]\right).
\end{split}
\end{equation}
\end{proposition}

\begin{proof}
Gathering Theorem $\ref{Theorem 3.1}$, $\eqref{eq: AB bound}$,
$\eqref{eq: C bound}$ and $\eqref{eq: D bound}$, we get $\eqref{eq:
3.12}$. $\eqref{eq: 3.13}$ can be obtained in a similar way.
\end{proof}

Using this and $\eqref{eq: 3.5}$ for the diagonal $L_n(F)$, we also
have
\begin{proposition}   \label{prop 3.3}
Under the same hypotheses of Proposition $\ref{prop 3.4}$ on $f$,
$\phi$ and $\{j_n\}$, there exist constants $\kappa_0$ (depending on
$\phi$ and $f$) and $n_0$(depending on F, f, $\phi$ and the sequence
$\{j_n\}$) such that, for all $\tau>0$ and for all $n\geq n_0$,
\begin{equation}
\begin{split} \label{eq: 3.10}
&\quad{\rm Pr}\left\{|W_n(F)|\geq\tau n2^{-{3\over 2}j_n}\right\}\\
&\leq\kappa_0\exp\left(-{1\over\kappa_0}\min\left[\frac{\tau^2}{\int_F
f^2(x)dx},2^{j_n/2}\tau,\tau^{2/3}n^{1/3}2^{-{j_n\over
3}},\tau^{1/2}n^{1/2}2^{-{j_n\over 4}},\tau^2n2^{-j_n},\tau
n2^{-{j_n\over 2}}\right]\right).
\end{split}
\end{equation}
In particular, if the sequence $2^{j_n}$ satisfies condition (B1)
and $\tau=\eta\sqrt{\log\log n}$, the first term dominates. For
every $\eta>0$ there exist $\kappa_0$ and $n_0$ as above such that
\begin{equation} \label{eq: 3.11}
{\rm Pr}\left\{|W_n(F)|\geq\eta n2^{-{3\over 2}j_n}\sqrt{\log\log
n}\right\}\leq\kappa_0\exp\left(-\frac{\eta^2\log\log
n}{\kappa_0\int_F f^2(x)dx}\right)
\end{equation}
for all $n\geq n_0$.
\end{proposition}

Now the three terms in the decomposition of $W_n(\rr)-W_{n,m}(\rr)$
in $\eqref{eq: 3.4}$ can be bounded. The first two are of the
U-statistics type, so Proposition $\ref{prop 3.4}$ is used to obtain
the estimation. The last one is a sum of mean zero i.i.d. r.v.'s and
can be dealt with by $\eqref{eq: 3.5}$.
\begin{lemma}  \label{Lemma Wn-Wnm}
Under the same hypotheses of Proposition $\ref{prop 3.4}$ on $f$,
$\phi$ and $\{j_n\}$, there exist a constant $\kappa_0$ (depending
on f and $\phi$) and $\eta>0$ such that, for all $\epsilon>0$,
$\sigma>0$, if $n$ is large enough (depending on $f$, $\phi$ and
$\{j_n\}$), and $m$ fixed is such that $0\leq m<n$,
\begin{equation}
{\rm Pr}\left\{|W_n(\rr)-W_{n,m}(\rr)|\geq\epsilon\sigma
n2^{-3j_n/2}\sqrt{2\log\log
n}\right\}\leq\kappa_0\exp\left(-\frac{\epsilon^2n^\eta}{\kappa_0}\right).
\end{equation}
\end{lemma}

\section{Moderate Deviations}

In this section, we'll prove a moderate deviation result for
$W_n([-M,M])$. This statistic can be approximated by a Gaussian
chaos due to the Koml\'{o}s-Major-Tusn\'{a}dy (KMT) theorem and the
Dvoretzky-Kiefer-Wolfowitz (DKW) inequalities. Then a moderate
deviation result in (GM) is used for the Gaussian chaos. $\phi$
satisfies both (S1) and (S2).

Let $F_n(t):={1\over n}\sum_{i=1}^n1(X_i\leq t)$ and $B_n$ be a
sequence of Brownian bridges. For all $x\in \rr$, set
\begin{equation}
\begin{split} \label{eq: E_n}
E_n(x):&=\sqrt{n2^{-j_n}}[f_{n,K}(x)-\ee f_{n,K}(x)]
=\sqrt{2^{j_n}\over n}\sum_{i=1}^n[K(2^{j_n}x,2^{j_n}X_i)-\ee K(2^{j_n}x,2^{j_n}X)]\\
&=\sqrt{n2^{j_n}}\int_\rr K(2^{j_n}x,2^{j_n}t)d[F_n(t)-F(t)].
\end{split}
\end{equation}
Let $K_{n,x}(t):=K(2^{j_n}x,2^{j_n}t)$ and
$\displaystyle{\mu_{K_{n,x}}(t)}$ be the Borel measure associated
with $K_{n,x}(t)$. Define the Gaussian process
\begin{equation}
\Gamma_n(x):=2^{j_n/2}\int_\rr[B_n(F(x))-B_n(F(t))]d\mu_{K_{n,x}}(t).
\end{equation}
We want to approximate
\begin{equation}  \label{eq: Wn[-M,M]}
{2^{3j_n/2}\over n}W_n([-M,M])=2^{j_n/2}\int_{-M}^M
\left[(E_n(t))^2-E(E_n(t))^2)\right]dt
\end{equation}
by a Gaussian chaos:
\begin{equation}  \label{eq: Gaussian}
2^{j_n/2}\int_{-M}^M
\left[(\Gamma_n(t))^2-E((\Gamma_n(t))^2)\right]dt.
\end{equation}
In order to apply the KMT theorem, we need an integration by parts
formula for $E_n(x)$. This requires us to check two conditions: (i)
$F_n(t)-F(t)$ and $K_{n,x}(t)$ are in the space $NBV$, where $NBV$
is defined by
\begin{equation}
NBV=\{G {\rm \ \ is \ \ of \ \ bounded \ \ variation},\ \ G {\rm\ \
is \ \ right \ \ continuous \ \ and} \  G(-\infty)=0\}.
\end{equation}
(ii) Almost surely, for fixed $N$, there are no points in $[-N,N]$
where $F_n(t)-F(t)$ and $K_{n,x}(t)$ are both discontinuous.

For any $m\in\nn$, let $\{-\infty<t_0<...<t_m=t\}$ be a partition
over $(-\infty,t)$. Then
\begin{equation}
\begin{split}
\sum_{l=1}^m|K_{n,x}(t_l)-K_{n,x}(t_{l-1})|
&\leq\sum_k|\phi(2^{j_n}x-k)|\sum_{l=1}^m|\phi(2^{j_n}t_l-k)-\phi(2^{j_n}t_{l-1}-k)|\\
&\leq\sum_k|\phi(2^{j_n}x-k)|\|\phi(2^{j_n}\cdot-k)\|_v.
\end{split}
\end{equation}
Since $\phi$ satisfies $\eqref{eq: phi}$ and (S2), we have, for
almost every $x$,
\begin{equation}
\|K_{n,x}\|_v\leq\sum_k|\phi(2^{j_n}x-k)|\|\phi\|_v :=C_\phi,
\end{equation}
where $C_\phi$ is a constant that depends only on the scaling
function $\phi$. The other conditions in (i) are obvious. To verify
(ii), we note that $K_{n,x}(t)$ could only have discontinuities at
dyadic points whereas $F_n(t)-F(t)$ could only have discontinuities
at $X_i$, $1\leq i\leq n$.

Then we apply an integration by parts formula (Ex. 3.34, Folland
1999) to the integral $\int_{[-N,N]} K_{n,x}(t)d[F_n(t)-F(t)]$ and
let $N\to\infty$. By dominated convergence, this gives
\begin{equation}
\int_\rr
K_{n,x}(t)d[F_n(t)-F(t)]+\int_\rr(F_n(t)-F(t))d\mu_{K_{n,x}}(t)=0.
\end{equation}
Moreover, since $\int_\rr d\mu_{K_{n,x}}(t)=0$,
\begin{equation}
E_n(x)=\sqrt{n2^{j_n}}\int_\rr[F(t)-F_n(t)-(F(x)-F_n(x))]d\mu_{K_{n,x}}(t).
\end{equation}
Now we are able to bound the difference between $\eqref{eq:
Wn[-M,M]}$ and $\eqref{eq: Gaussian}$. We set $\alpha_n(t):=\sqrt
n\left[F_n(t)-F(t)\right]$ and
$D_n:=\sup_{-\infty<t<\infty}|\alpha_n(t)-B_n(F(t))|$. We have
\begin{equation}
\begin{split} \label{eq: DnM}
D_n(M):&=\left|{2^{3j_n/2}\over n}W_n([-M,M])-2^{j_n/2}\int_{-M}^M
\left((\Gamma_n(t))^2-E((\Gamma_n(t))^2)\right)dt\right|\\
&\leq2^{j_n}\cdot4MD_nC_\phi\esssup_x(|E_n(x)|+|\Gamma_n(x)|)\\
&\leq2^{3j_n/2}8MD_n(\|\alpha_n\|_{\infty}+\|B_n\|_{\infty})C_\phi^2.
\end{split}
\end{equation}
We use the KMT theorem for $D_n$ and the DKW inequalities for
$\|\alpha_n\|_{\infty}$ and $\|B_n\|_{\infty}$.
\begin{theorem}(Koml\'{o}s, Major, Tusn\'{a}dy, 1975)
There exists a probability space $(\Omega,\mathcal {A},P)$ with
i.i.d random variables $X_1,X_2,...$, with density f and a sequence
of Brownian bridges $B_1,B_2,...$, such that, for all $n\geq1$ and
$x\in \rr$,
\begin{equation}   \label{eq: D_n}
{\rm Pr}\left\{D_n\geq n^{-1/2}(a\log n+x)\right\}\leq b\exp(-cx),
\end{equation}
where a,b and c are positive constants that do not depend on n, x or
f.
\end{theorem}
The DKW inequalities (Dvoretzky, Kiefer, Wolfowitz, 1956; or see
Shorack and Wellner, 1986) give that, for every $z>0$,
\begin{equation} \label{eq: dkw}
{\textrm Pr}\left\{\|\alpha_n\|_{\infty}>z\right\}\leq2\exp(-2z^2),\
\ {\textrm Pr}\left\{\|B_n\|_{\infty}>z\right\}\leq2\exp(-2z^2).
\end{equation}
We arrive at the following proposition.
\begin{proposition}  \label{prop 4.2}
Assuming the scaling function $\phi$ satisfies (S1), (S2) and
${j_n}$ satisfies (B1), for any $\gamma>0$ there exists
$C_{M,\phi}>0$ such that
\begin{equation} \label{eq: Cmphi}
{\rm Pr}\left\{D_n(M)\geq\frac{C_{M,\phi}(\log
n)^2}{2^{-3j_n/2}\sqrt n}\right\}\leq n^{-\gamma}
\end{equation}
for all $n>n_0(\gamma)$.
\end{proposition}

\begin{proof}
For $\gamma>0$, take $x=2\gamma\log n/c$ in $\eqref{eq: D_n}$. If
$n$ is sufficiently large depending on $\gamma$,
\begin{equation}
{\textrm Pr}\left\{D_n\geq\frac{1}{\sqrt
n}\left(a+\frac{2\gamma}{c}\right)\log n\right\}\leq b
\exp\left(-2\gamma\log n\right)\leq\frac{1}{2}n^{-\gamma}.
\end{equation}
From DKW inequalities $\eqref{eq: dkw}$, it is easy to see that for
$n$ large enough,
\begin{equation}
{\textrm Pr}\left\{\|\alpha_n\|_\infty+\|B_n\|_\infty>\frac{\log
n}{a+2\gamma/c}\right\}\leq\frac{1}{2}n^{-\gamma}.
\end{equation}
Combining these with $\eqref{eq: DnM}$, we get
\begin{equation}
\begin{aligned}
&\quad{\textrm Pr}\left\{D_n(M)\geq \frac{8MC_\phi^2(\log
n)^2}{2^{-3j_n/2}\sqrt{n}}\right\}\\
&\leq{\textrm Pr}\left\{D_n\geq\frac{1}{\sqrt
n}\left(a+\frac{2\gamma}{c}\right)\log n\right\}+{\textrm
Pr}\left\{\|\alpha_n\|_\infty+\|B_n\|_\infty>\frac{\log
n}{a+2\gamma/c}\right\}\\
&\leq n^{-\gamma}.
\end{aligned}
\end{equation}
Setting $C_{M,\phi}=8MC_\phi^2$ yields $\eqref{eq: Cmphi}$.
\end{proof}

It is easier to obtain a moderate deviation result for
$2^{j_n/2}\int_{-M}^M ((\Gamma_n(t))^2-E((\Gamma_n(t))^2))dt$ than
for $2^{3j_n/2}W_n([-M,M])/n$. For the former we can adapt the
method in (GM) where they obtain a moderate deviation result for
similar random variables by adapting a method of Pinsky (Pinsky,
1966) to prove the LIL for sums of random variables with finite
moments higher than 2. It is a well-known fact that $\int_{-M}^M
((\Gamma_n(t))^2-E((\Gamma_n(t))^2))dt$ can be written as a sum of
weighted, centered chi-squared random variables(e.g., Proposition
4.3, GM, 2004). Recall the operator $\mathcal{R}_{n,F}$ defined in
$\eqref{eq: operator}$. Let
$\lambda_{n,1}\geq\lambda_{n,2}\geq\ldots\geq0$ be the eigenvalues
of the operator $\mathcal{R}_{n,F}$ with $F=[-M,M]$. $Z_k$ are i.i.d
$\mathcal{N}(0,1)$. We then have
\begin{equation}   \label{eq: gamma}
\int_F\left[(\Gamma_n(t))^2-\ee(\Gamma_n(t))^2\right]dt=\sum_{k=1}^\infty\lambda_{n,k}(Z_k^2-1).
\end{equation}
The limiting variance is calculated using Lemmas $\ref{Lemma 2.5}$,
$\ref{Lemma 2.4}$:
\begin{equation}
\begin{split}   \label{eq: modlim}
&\quad\lim_{n\to\infty}2^{j_n}\ee\left[\int_{-M}^M\left((\Gamma_n(t))^2-E(\Gamma_n(t))^2\right)dt\right]^2\\
&=\lim_{n\to\infty}2\cdot2^{j_n}\sum_{k=1}^\infty\lambda_{n,k}^2\\
&=\lim_{n\to\infty}2\cdot2^{j_n}\int_{-M}^M\int_{-M}^MR_n^2(s,t)dsdt\\
&=2\int_{-M}^Mf^2(x)dx=:\sigma^2(M).
\end{split}
\end{equation}
Set
$b_n:=\left(\lambda_{n,1}/\sqrt{\sum_{k=1}^\infty\lambda_{n,k}^2}\right)^\eta$
for some $0<\eta\leq1$ and
\begin{equation} V_n(M):={2^{j_n/2}\over
\sigma(M)}\int_{-M}^M\Bigl((\Gamma_n(t))^2-E(\Gamma_n(t))^2\Bigl)dt.
\end{equation}
Using $\eqref{eq: gamma}$ and a modification of Pinsky's method, we
have a moderate deviation for $V_n(M)$, which is parallel to
$(4.15)$, (GM). For any sequence $a_n$ converging to infinity at the
rate $a_n^2+\log b_n\to-\infty$ and for all $0<\epsilon<1$,
\begin{equation}  \label{eq: VnM}
\exp\left(-\frac{a_n^2(1+\epsilon)}{2}\right)\leq{\rm Pr}\left\{\pm
V_n(M)\geq
a_n\right\}\leq\exp\left(-\frac{a_n^2(1-\epsilon)}{2}\right)
\end{equation}
if $n$ is large enough depending on $\epsilon$.

We can use this result, the triangle inequality and Proposition
$\ref{prop 4.2}$ to obtain:
\begin{proposition} \label{prop 4.7}
Let $a_n=C\sqrt{2\log\log n},\ 0<C<\infty$. Under the hypotheses of
Proposition $\ref{prop 4.2}$, and further assuming that $f$
satisfies condition (f) and that $\int_{-M}^Mf^2(x)dx>0$, then we
have a two-sided inequality,
\begin{equation}  \label{eq: 4.16}
\exp\left(-\frac{a_n^2(1+\epsilon)}{2}\right)-{1\over n^2}\leq{\rm
Pr}\left\{\pm {2^{3j_n/2}\over \sigma(M)n}W_n([-M,M])\geq
a_n\right\}\leq\exp\left(-\frac{a_n^2(1-\epsilon)}{2}\right)+{1\over
n^2}
\end{equation}
for all $0<\epsilon<1$ and $n$ large enough (depending on $M$ and
$\epsilon$).
\end{proposition}

\section{Main Proofs}
\subsection{Theorem $\ref{Theorem 5.1}$}
\begin{proof}
We show that $J_n=\bar{J}_n$, where $\bar{J}_n$ is defined in
$\eqref{eq: defJnb}$. Since we have,
\begin{equation}
J_n=\int_\rr f_{n,K}^2-\ee f_{n,K}^2-2f_{n,K}f+2f\ee f_{n,K},
\end{equation}
and
\begin{equation}
\bar{J}_n=\int_\rr f_{n,K}^2-2f_{n,K}\ee f_{n,K}-\ee
f_{n,K}^2+2\left(\ee f_{n,K}\right)^2.
\end{equation}
It remains to show that the difference
\begin{equation}
J_n-\bar{J}_n=2\int_\rr (f-\ee f_{n,K})(\ee f_{n,K}-f_{n,K})=0.
\end{equation}
$\ee f_{n,K}-f_{n,K}$ is a linear combination of $\{\phi_{0k}\}$ and
$\{\psi_{jk}\}$, $0\leq j\leq j_n-1$, whereas $f-\ee f_{n,K}$ is a
linear combination of $\{\psi_{jk}\}$, $j\geq j_n$. By orthogonality
of $\{\phi_{0k},\psi_{jk}\}$, we have $J_n-\bar{J}_n=0$. Thus the
proof of Theorem $\ref{Theorem 5.1}$ reduces to proving that
\begin{equation}
\limsup_{n\to\infty}\pm\frac{n2^{-j_n/2}}{\sigma\sqrt{2\log\log
n}}\bar{J}_n=1, \ \ \ \ a.s.
\end{equation}
By $\eqref{eq: Jnbar}$, this is equivalent to
\begin{equation}  \label{eq: lb}
\limsup_{n\to\infty}\pm\frac{2^{3j_n/2}W_n(\rr)}{n\sigma\sqrt{2\log\log
n}}=1.
\end{equation}
Since we have analogous variance computation, tail estimation and
moderate deviation results to those for the kernel density
estimator, the proof is the same as in Theorem 5.1, (GM). We give an
outline of the proof but readers should refer to (GM) for details.

(i) Proof of the lower bound: Lemma $\ref{Lemma Wn-Wnm}$ and
Borel-Cantelli implies that the random variable
$\displaystyle{\limsup_n\frac{W_n(\rr)}{\sigma
n2^{-3j_n/2}\sqrt{2\log\log n}}}$ is measurable with respect to the
tail $\sigma$-algebra of ${X_i}$. We assume the lower bound is not
true. In particular, we choose $r_k=k^k$, then there exists $c<1$
s.t.
\begin{equation}
\limsup_k\frac{W_{r_k}(\rr)}{\sigma
r_k2^{-3j_{r_k}/2}\sqrt{2\log\log r_k}}=c \ \ \ \ a.s.
\end{equation}
The proof of Lemma $\ref{Lemma Wn-Wnm}$ also applies to
$W_{r_k}(\rr)-W_{{r_k},{r_{k-1}}}(\rr)$ since $r_k/r_{k-1}\geq k$.
And we have
\begin{equation}
\frac{|W_{r_k}(\rr)-W_{{r_k},{r_{k-1}}}(\rr)|}{r_k\sigma2^{-3j_{r_k}/2}\sqrt{2\log\log
r_k}}\to0   \quad   a.s.
\end{equation}
Thus
\begin{equation}
\limsup_k\frac{W_{{r_k},{r_{k-1}}}(\rr)}{\sigma
r_k2^{-3j_{r_k}/2}\sqrt{2\log\log r_k}}=c \ \ \ \ a.s.
\end{equation}
By Borel-Cantelli, there exists $c'$ satisfying $c<c'<1$, s.t.
\begin{equation}   \label{eq: 5.5}
\sum_k{\textrm Pr}\left\{\frac{W_{{r_k},{r_{k-1}}}(\rr)}{\sigma
r_k2^{-3j_{r_k}/2}\sqrt{2\log\log r_k}}\geq c'\right\}<\infty.
\end{equation}
Set $m_k:=r_k-r_{k-1}$ and define
\begin{equation}
W_{m_k}'(\rr):=\int_\rr\left(\sum_{i=1}^{r_k-r_{k-1}}\bar{K}(2^{jr_k}t,2^{jr_k}X_i)\right)^2dt
-\ee\int_\rr\left(\sum_{i=1}^{r_k-r_{k-1}}\bar{K}(2^{jr_k}t,2^{jr_k}X_i)\right)^2dt.
\end{equation}
Since $W_{m_k}'(\rr)$ and $W_{r_k,r_{k-1}}(\rr)$ have the same
distribution, $\eqref{eq: 5.5}$ holds with
$W_{{r_k},{r_{k-1}}}(\rr)$  replaced by $W_{m_k}'(\rr)$. This and
$m_k/r_k\to1$ imply that there exists $c''$ satisfying $c'<c''<1$,
s.t.
\begin{equation} \label{eq: 5.5'}
\sum_k{\textrm Pr}\left\{W_{m_k}'(\rr)\geq c''\sigma
m_k2^{-3j_{r_k}/2}\sqrt{2\log\log m_k}\right\}<\infty.
\end{equation}
We choose $M$ large enough so that $\int_{[-M,M]^C}f^2(x)dx<(\delta
c''\sigma)^2/\kappa_0,$ where $\kappa_0$ is the constant in
$\eqref{eq: 3.11}$. $W_{m_k}'(\rr)$ can be split into
$W_{m_k}'([-M,M])$ and $W_{m_k}'([-M,M]^C)$. $\eqref{eq: 3.11}$ is
used for $W_{m_k}'([-M,M]^C)$ and Proposition $\ref{prop 4.7}$ for
$W_{m_k}'([-M,M])$. Then we would reach a contradiction to
$\eqref{eq: 5.5'}$ and thus prove the lower bound.

(ii) Proof of the upper bound: We shall first use conditions $(B1)$
and $(B2)$ to introduce a blocking and reduce $W_n(\rr)$ to
$W_{n_k}(\rr)$ for the sequence $n_k:=\min\{n\in\nn:
n\geq\lambda_k\}$. $n_k$ satisfies the same properties as
$\lambda_k$ does. $I_k$ is the block defined by
$I_k:=[n_k,n_{k+1})\cap\nn.$ $I_k$ is nonempty for $k\geq k_0$.

By Borel-Cantelli, it suffices to show that, for every $\delta>0$,
\begin{equation}   \label{eq: 5.10}
\sum_{k\geq k_0}{\textrm Pr}\left\{\max_{n\in
I_k}|W_n(\rr)|>(1+\delta)\sigma n_k2^{-3j_{n_k}/2}\sqrt{2\log\log
n_k}\right\}<\infty.
\end{equation}
We will prove that for every $\tau>0$,
\begin{equation}   \label{eq: 5.11}
\sum_{k\geq k_0}{\textrm Pr}\left\{\max_{n\in
I_k}|W_n(\rr)-W_{n_k}(\rr)|>\tau\sigma
n_k2^{-3j_{n_k}/2}\sqrt{2\log\log n_k}\right\}<\infty.
\end{equation}
For $n\in I_k$, similar to $\eqref{eq: 3.4}$, we have
\begin{equation}
\begin{split}  \label{eq: 5.13}
W_n(\rr)-W_{n_k}(\rr)=&2\sum_{i=1}^{n_k}\sum_{j=n_k+1}^n
H_{n_k}(X_i,X_j)+\sum_{n_k<i\neq j\leq
n}H_{n_k}(X_i,X_j)\\
&+\sum_{i=n_k+1}^n(H_{n_k}(X_i,X_i)-\ee H_{n_k}(X_i,X_i)).
\end{split}
\end{equation}
$H_n$ is replaced by $H_{nk}$ since $\{2^{-j_n}\}$ is constant for
$n\in I_k$ by hypothesis. We will apply Montgomery-Smith maximal
inequality (Montgomery-Smith, 1993) to the first and the last
summands directly: If $X_i$ are $i.i.d$ r.v.'s taking values in a
Banach space and $\|\cdot\|$ is a norm in the Banach space, then
\begin{equation}
{\textrm Pr}\left\{\max_{1\leq k\leq
n}\left\|\sum_{i=1}^kX_i\right\|>t\right\}\leq9{\textrm
Pr}\left\{\left\|\sum_{i=1}^nX_i\right\|>{t\over 30}\right\}.
\end{equation}
However, the second summand is not a sum of i.i.d random variables.
A decoupling inequality (e.g., de la Pe\~{n}a and Gin\'{e}, 1999,
Theorem 3.4.1) is used to transform it into independent variables,
i.e., $\sum_{n_k<i\neq j\leq n}H_{n_k}(X_i^{(1)}, X_j^{(2)})$, where
$X_i^{(1)}$ and $X_j^{(2)}$, $i,j\in \nn$ are $i.i.d.$ copies of
$X_1$. Then we add the diagonal, apply Montgomery-Smith inequality
twice and subtract the diagonal at last. We will be able to reduce
$\eqref{eq: 5.11}$ to proving that, for every $\tau>0$,
\begin{equation} \label{eq: 5.16}
\sum_{k\geq k_0}{\textrm
Pr}\left\{\left|\sum_{j=n_k+1}^{n_{k+1}-1}\sum_{i=1}^{n_k}H_{n_k}(X_i,X_j)\right|>\tau\sigma
n_k2^{-3j_{n_k}/2}\sqrt{2\log\log n_k}\right\}<\infty,
\end{equation}
\begin{equation} \label{eq: 5.19}
\sum_{k\geq k_0}{\textrm
Pr}\left\{\left|\sum_{i=n_k+1}^{n_{k+1}-1}(H_{n_k}(X_i,X_i)-\ee
H_{n_k}(X_i,X_i))\right|> \tau\sigma
n_k2^{-3j_{n_k}/2}\sqrt{2\log\log n_k}\right\}<\infty,
\end{equation}
and
\begin{equation} \label{eq: 5.17}
\sum_{k\geq k_0}{\textrm
Pr}\left\{\left|\sum_{j=n_k+1}^{n_{k+1}-1}\sum_{i=n_k+1}^{n_{k+1}-1}H_{n_k}(X_i^{(1)},
X_j^{(2)})\right|>\tau\sigma n_k2^{-3j_{n_k}/2}\sqrt{2\log\log
n_k}\right\}<\infty,
\end{equation}
\begin{equation} \label{eq: 5.18}
\sum_{k\geq k_0}{\textrm Pr}\left\{\left|\sum_{i=n_k+1}^
{n_{k+1}-1}H_{n_k}(X_i^{(1)}, X_i^{(2)})\right|>\tau\sigma
n_k2^{-3j_{n_k}/2}\sqrt{2\log\log n_k}\right\}<\infty.
\end{equation}
$\eqref{eq: 5.16}$, $\eqref{eq: 5.19}$ come from the first and last
summands in $\eqref{eq: 5.13}$ whereas $\eqref{eq: 5.17}$,
$\eqref{eq: 5.18}$ come from the second summand. We apply
Bernstein's inequality to $\eqref{eq: 5.19}$ and $\eqref{eq: 5.18}$.
Proposition $\ref{prop 3.4}$ will take care of $\eqref{eq: 5.16}$
and $\eqref{eq: 5.17}$. Therefore, $\eqref{eq: 5.11}$ is proved.
Thus $\eqref{eq: 5.10}$ is reduced to showing that for every
$\delta>0$,
\begin{equation} \label{eq: 5.12}
\sum_{k\geq k_0}{\textrm Pr}\left\{|W_{n_k}(\rr)|>(1+\delta)\sigma
n_k2^{-3j_{n_k}/2}\sqrt{2\log\log n_k}\right\}<\infty.
\end{equation}

The second step is to reduce $W_{nk}(\rr)$ to $W_{nk}([-M,M])$ for
some $M$ large enough. Given $\delta>0$, there exists $M<\infty$
such that $\int_{[-M,M]^C}f^2(x)dx<\delta^2\sigma^2/(4\kappa_0)$,
where $\kappa_0$ is the constant in inequality $\eqref{eq: 3.11}$.
Application of $\eqref{eq: 3.11}$ gives that, from some $k$ on,
\begin{equation} \label{eq: 5.21}
\Pr\left\{|W_{n_k}([-M,M]^C)|>{\delta\over2}\sigma
n_k2^{-3j_{n_k}/2}\sqrt{2\log\log
n_k}\right\}\leq\kappa_0\exp\left(-2\log\log n_k\right),
\end{equation}
where the right hand side is the general term of a convergent
series. Let $\epsilon$ be so small that
$(1+\delta/2)^2(1-\epsilon)>1$. Now we use $\eqref{eq: 4.16}$ to
obtain that, for $n_k$ large enough,
\begin{equation}
\begin{split}
&\quad{\textrm Pr}\left\{|W_{n_k}([-M,M])|>(1+{\delta\over2})\sigma
n_k2^{-3j_{n_k}/2}\sqrt{2\log\log n_k}\right\}\\
&\leq{\textrm
Pr}\left\{|W_{n_k}([-M,M])|>(1+{\delta\over2})\sigma(M)
n_k2^{-3j_{n_k}/2}\sqrt{2\log\log n_k}\right\}\\
&\leq\exp\left(-(1+\delta/2)^2(1-\epsilon)\log\log
n_k\right)+\frac{1}{n_k^2}\ ,
\end{split}
\end{equation}
which is also the general term of a convergent series. Hence the
series $\eqref{eq: 5.12}$ converges for every $\delta>0$.

\end{proof}

\subsection{Theorem $\ref{Theorem BE}$}
\begin{proof} Without loss of generality, we will assume that, for all $n$,
there exist constants $C_1$ and $C_2$, such that
$C_1n^{\delta}\leq2^{j_n}\leq C_2n^{\delta}.$ Proving Theorem
$\ref{Theorem BE}$ is equivalent to proving that
\begin{equation}
\sup_t|\Pr\{n2^{-j_n/2}\bar{J}_n\leq t\}-\Pr\{\sigma Z\leq t\}|\leq
C(n^{-3\delta/16}\vee n^{-\alpha\delta}\sqrt{\log n}).
\end{equation}
By $\eqref{eq: Jnbar}$ and $\eqref{eq: Wn dec}$, we have that
\begin{equation}
n2^{-j_n/2}\bar{J}_n/\sigma=\frac{2^{3j_n/2}}{n\sigma}W_n(\rr)
=\frac{2^{3j_n/2}}{n\sigma}U_n(\rr)+\frac{2^{3j_n/2}}{n\sigma}L_n(\rr).
\end{equation}
Using the triangle inequality, we can obtain an upper bound and a
lower bound for this statistic. For an arbitrary positive sequence
$\epsilon_{1,n}$,
\begin{equation}
\begin{aligned}    \label{eq: BE}
&\quad\sup_t\left|\Pr\{n2^{-j_n/2}\bar{J}_n/\sigma\leq t\}-\Pr\{Z\leq t\}\right|\\
&\leq\sup_t\left|\Pr\left\{\frac{2^{3j_n/2}}{n\sigma}U_n(\rr)\leq t\right\}-\Pr\{Z\leq t\}\right|\\
&\quad+\Pr\left\{\frac{2^{3j_n/2}}{n\sigma}\left|L_n(\rr)\right|>\epsilon_{1,n}\right\}
+\sup_t\Pr\left\{t-\epsilon_{1,n}<Z\leq t+\epsilon_{1,n}\right\}.
\end{aligned}
\end{equation}
It's easy to bound the last term:
\begin{equation} \label{eq: est nor}
\sup_t\Pr\left\{t-\epsilon_{1,n}<Z\leq
t+\epsilon_{1,n}\right\}<\epsilon_{1,n}.
\end{equation}
By $\eqref{eq: 3.5}$, for $0<\epsilon_{1,n}\leq1$ so that
$\epsilon_{1,n}^2\leq\epsilon_{1,n}$,
\begin{align}
\Pr\left\{\left|L_n(\rr)\right|>\sigma\epsilon_{1,n}n2^{-3j_n/2}\right\}
&\leq
C\exp\left(-\frac{1}{C}\min\left(\sigma^2\epsilon_{1,n}^2n2^{-j_n},\sigma\epsilon_{1,n}
n2^{-j_n/2}\right)\right)\notag\\
&\leq C\exp\left(-\frac{1}{C}\epsilon_{1,n}^2n^{1-\delta}\right),
\end{align}
where $C$ depends on both $\phi$ and $f$, $\delta\in(0,1/3)$. We may
take $\epsilon_{1,n}=n^{-1/3}$ to obtain
\begin{equation}    \label{eq: be1}
\Pr\left\{\left|L_n(\rr)\right|>\sigma\epsilon_{1,n}n2^{-3j_n/2}\right\}\leq
C\exp\left(-\log n\right)=Cn^{-1}
\end{equation}
when $n$ is large enough. Using $\eqref{eq: est nor}$, we get
\begin{equation}   \label{eq: be2}
\sup_t\Pr\left\{t-\epsilon_{1,n}<Z\leq t+\epsilon_{1,n}\right\}\leq
n^{-1/3}.
\end{equation}

To control the first term in $\eqref{eq: BE}$, we will approximate
$2^{3j_n/2}U_n(\rr)/(n\sigma)$ by $S_{nn}$, which is defined below.
We set
\begin{equation} \label{eq: Unn}
U_{nn}:=\sum_{i=2}^n\sum_{j=1}^{i-1}H_n(X_i,X_j), \
s_n^2:=\ee(U_{nn}^2),
\end{equation}
and
\begin{equation}  \label{eq: Xni}
X_{ni}:=\sum_{j=1}^{i-1}\frac{H_n(X_i,X_j)}{s_n}, \
S_{nk}:=\sum_{i=2}^kX_{ni},
\end{equation}
then
\begin{equation}  \label{eq: Snn}
S_{nn}=\sum_{i=2}^n\sum_{j=1}^{i-1}\frac{H_n(X_i,X_j)}{s_n}.
\end{equation}
Analogous to $\eqref{eq: BE}$, for any positive sequence
$\epsilon_{2,n}$,
\begin{equation}
\begin{aligned}    \label{eq: BE2}
&\sup_t\left|\Pr\left\{\frac{2^{3j_n/2}}{n\sigma}U_n(\rr)\leq t\right\}-\Pr\{Z\leq t\}\right|\\
\leq&\sup_t\left|\Pr\left\{S_{nn}\leq t\right\}-\Pr\{Z\leq t\}\right|\\
&+\Pr\left\{\left|\frac{2^{3j_n/2}}{n\sigma}U_n(\rr)-S_{nn}\right|>\epsilon_{2,n}\right\}
+\sup_t\Pr\left\{t-\epsilon_{2,n}<Z\leq t+\epsilon_{2,n}\right\}.
\end{aligned}
\end{equation}
By $\eqref{eq: 3.3}$,
\begin{align}     \label{eq: Un-Sn}
\Pr\left\{\left|\frac{2^{3j_n/2}}{n\sigma}U_n(\rr)-S_{nn}\right|>\epsilon_{2,n}\right\}
=\Pr\left\{\left|\sum_{1\leq i\neq j\leq
n}H_n(X_i,X_j)\right|>\frac{\epsilon_{2,n}}{d_n}\right\},
\end{align}
where
$\displaystyle{d_n=\left|\frac{2^{3j_n/2}}{n\sigma}-\frac{1}{2s_n}\right|}$.
We then estimate the order of $d_n$. Recall the definition of
$R_n(s,t)$ in $\eqref{eq: Rn}$ and set
$e_n:=\left(2^{j_n}\int_{\rr^2} R_n^2(s,t)dsdt\right)^{1/2}$. Using
the definition of $s_n^2$ and Fubini's theorem, we get
\begin{equation}   \label{eq: sn2}
s_n^2=\sum_{i=2}^n\sum_{j=1}^{i-1}\ee
H_n^2(X_i,X_j)=\frac{n(n-1)}{2}2^{-3j_n}e_n^2.
\end{equation}
Plugging it into $d_n$ and using a triangle inequality, we then have
\begin{align}
d_n\leq&C2^{3j_n/2}\left|\frac{1}{n\sqrt{\int
f^2(x)dx}}-\frac{1}{\sqrt{n(n-1)\int
f^2(x)dx}}\right|\notag\\
&+C2^{3j_n/2}\left|\frac{1}{\sqrt{n(n-1)\int
f^2(x)dx}}-\frac{1}{\sqrt{n(n-1)}e_n}\right|.
\end{align}
Since $2^{j_n}\leq Cn^\delta$ for some $\delta\in(0,1/3)$ and
$1/\sqrt{n(n-1)}-1/n\leq n^{-2}$ when $n\geq 2$, the first term is
bounded by $Cn^{3\delta/2-2}$. Corollary $\ref{Lemma Rn2}$ gives
that $\left|e_n^2-\int_\rr f^2(x)dx\right|\leq
C(n^{-\delta/2}+n^{-\delta\alpha})$. The second term is bounded by
$Cn^{3\delta/2-1}(n^{-\delta/2}+n^{-\delta\alpha})$ when $n$ is
large enough. Combining the two terms, $d_n\leq
Cn^{3\delta/2-1}\left(n^{-\delta/2}+n^{-\delta\alpha}\right)$, where
$C$ depends on $f, \{j_n\}$ and $\phi$. Taking
$\epsilon_{2,n}=n^{-\delta(\frac{1}{2}\wedge\alpha)}\sqrt{\log n}$
and using $\eqref{eq: Un-Sn}$, Proposition $\ref{prop 3.4}$, we
obtain
\begin{equation}   \label{eq: be3}
\Pr\left\{\left|\frac{2^{3j_n/2}}{n\sigma}U_n(\rr)-S_{nn}\right|>\epsilon_{2,n}\right\}\leq
\kappa_0\exp\left(-\log n\right)=Cn^{-1}
\end{equation}
when $n$ is large enough. Consequently,
\begin{equation}   \label{eq: be4}
\sup_t\Pr\left\{t-\epsilon_{2,n}<Z\leq t+\epsilon_{2,n}\right\}\leq
n^{-\delta(\frac{1}{2}\wedge\alpha)}\sqrt{\log n}.
\end{equation}
We then deal with $\sup_t \left|\Pr\left\{S_{nn}\leq t\right\}-
\Pr\{Z\leq t\}\right|$. Let $\mathcal{F}_i$ be the $\sigma$-field
generated by $\{X_1, X_2, ..., X_i\}$ for $i=1,2,...$.  We first
observe that, by the definitions in $\eqref{eq: Unn}$-$\eqref{eq:
Snn}$,
\begin{equation}
\mu_{ni}:=\ee(X_{ni}|\mathcal{F}_{i-1})=0,
\end{equation}
and thus $S_{nk}$ is a martingale with respect to $\mathcal{F}_{k}$.
We will use the result of Erickson, Quine and Weber (1979) to derive
a bound for $\sup_t\left|\Pr\left\{S_{nn}\leq t\right\}-\Pr\{Z\leq
t\}\right|$. For $i\geq2$, let $X_{ni}':=X_{ni}-\mu_{ni}$,
$\sigma_{ni}^2:=\ee\left({X_{ni}'}^2|\mathcal{F}_{i-1}\right)$ and
$\sigma_n^2:=\sum_{i=2}^n\sigma_{ni}^2$. Also define
$Y_{ni}:=\sum_{j=1}^{i-1}H_n(X_i,X_j)$ and
$V_n^2:=\sum_{i=2}^n\ee\left(Y_{ni}^2|\mathcal{F}_{i-1}\right)$.
\begin{theorem}[Erickson, Quine, Weber, 1979]
Given $X=\{X_{ni}, i=2,...,n; n=1,2,...\}$ and
$\mathcal{F}=\{\mathcal{F}_{i}, i=1,2,...\}$, let
$S_{nn}:=\sum_{i=2}^nX_{ni}$. If $\mu_{ni}=0$ for all $n,i$, then
for $\eta\in(0,1]$, there exists a constant $C$,
\begin{equation}     \label{eq: mart}
\sup_t\left|\Pr\left\{S_{nn}\leq t\right\}-\Pr\{Z\leq t\}\right|\leq
C\left\{\sum_{i=2}^n\ee|X_{ni}|^{2+\eta}+\ee|1-\sigma_n^2|^{1+\eta/2}\right\}^{1/(3+\eta)}.
\end{equation}
\end{theorem}
Consider the second term:
\begin{equation}
\begin{aligned}
&\ee\left|1-\sigma_n^2\right|^2=s_n^{-4}\ee\left|s_n^2-V_n^2\right|^2\leq
s_n^{-4}\ee(V_n^4).
\end{aligned}
\end{equation}
Set $G_n(x,y)=\ee\left(H_n(X_1,x)H_n(X_1,y)\right)$, then by the
proof of Theorem 1, Hall (1984),
\begin{equation}
\ee(V_n^4)\leq
C\left(n^4\ee{G_n^2(X_1,X_2)}+n^3\ee{G_n^2(X_1,X_1)}\right)\leq
C\left(n^4\ee{G_n^2(X_1,X_2)}+n^3\ee{H_n^4(X_1,X_2)}\right).
\end{equation}
By $\eqref{eq: sn2}$ and Corollary $\ref{Lemma Rn2}$, $s_n^4\asymp
n^{4-6\delta}$. The calculations in Theorem 1, Zhang and Zheng
(1999) can be applied here directly. $H_n(x,y)$ defined in
$\eqref{eq: Hn}$ is off by a scaling constant $2^{-2j_n}n^2$ from
their definition.
\begin{equation}    \label{eq: Hn4}
\ee{H_n^4(X_1,X_2)}=\left(2^{-2j_n}{n^2}\right)^4O(2^{3j_n}/n^8)=O(2^{-5j_n})=O(n^{-5\delta}),
\end{equation}
and
\begin{equation}
\ee{G_n^2(X_1,X_2)}=\left(2^{-2j_n}{n^2}\right)^4O(2^{j_n}/n^8)=O(2^{-7j_n})=O(n^{-7\delta}).
\end{equation}
Combining these estimates and using H\"{o}lder inequality, we see
\begin{equation}  \label{eq: var df}
\ee|1-\sigma_n^2|^{1+\eta/2}\leq Cn^{-\delta(2+\eta)/4}.
\end{equation}
For the first term in $\eqref{eq: mart}$, we observe that
\begin{equation}
\sum_{i=2}^n\ee|X_{ni}|^{2+\eta}\leq\sum_{i=2}^n\frac{1}{s_n^{2+\eta}}
\left(\ee\left|\sum_{j=1}^{i-1}H_n(X_i,X_j)\right|^3\right)^{(2+\eta)/3}.
\end{equation}
Let $\ee_i$ denote the expectation with respect to $X_i$ and
$\ee_{i'}$ denote the expectation with respect to $X_1,...,X_{i-1}$.
We can apply a Hoffmann-Jorgensen type inequality with respect to
$\ee_{i'}$(Theorem 1.5.13, de la Pe\~{n}a and Gin\'{e}, 1999),
\begin{equation}
\ee\left|\sum_{j=1}^{i-1}H_n(X_i,X_j)\right|^3\leq
C\ee_{i}\left\{\ee_{i'}\max_{1\leq j\leq
i-1}\left|H_n(X_i,X_j)\right|^3+\left(\ee_{i'}\left(\sum_{j=1}^{i-1}H_n(X_i,X_j)\right)^2\right)^{3/2}\right\}.
\end{equation}
The first term can be bounded using $\eqref{eq: 2.16}$. For the
second one, we use Jensen's inequality, H\"{o}lder inequality and
$\eqref{eq: Hn4}$ to get
\begin{equation}
\ee_{i}\left(\ee_{i'}\left(\sum_{j=1}^{i-1}H_n(X_i,X_j)\right)^2\right)^{3/2}
=\ee_{i}\left((i-1)\ee_1H_n^2(X_1,X_2)\right)^{3/2}\leq
C(i-1)^{3/2}n^{-15\delta/4}.
\end{equation}
These inequalities and $\sum_{i=2}^ni^{(2+\eta)/2}\leq
Cn^{2+\eta/2}$ lead to
\begin{equation} \label{eq: Xni bd}
\sum_{i=2}^n\ee|X_{ni}|^{2+\eta} \leq
C{n^{(3\delta/2-1)(2+\eta)}}n^{-\delta(2+\eta)}\sum_{i=2}^n\max(1,i^{3/2}n^{-3\delta/4})^{(2+\eta)/3}
\leq Cn^{\delta/2+\eta\delta/4-\eta/2}.
\end{equation}
Gathering $\eqref{eq: mart}, \eqref{eq: var df}$ and $\eqref{eq: Xni
bd}$ and noting that the bound is minimized when $\eta=1$, we arrive
at
\begin{equation}
\sup_t\left|\Pr\left\{S_{nn}\leq t\right\}-\Pr\{Z\leq t\}\right|\leq
C\max\left(n^{3\delta/16-1/8},n^{-3\delta/16}\right)\leq
Cn^{-3\delta/16}.
\end{equation}
Putting together the last inequality with $\eqref{eq: BE},
\eqref{eq: be1}, \eqref{eq: be2}, \eqref{eq: BE2}, \eqref{eq: be3}$
and $\eqref{eq: be4}$, we conclude that when $n$ is large enough
(depending on $f$ and $\phi$), there exists a constant C (depending
on $f$, $\phi$ and $\{j_n\}$),
\begin{equation}
\begin{aligned}
\sup_t|\Pr\{n2^{-j_n/2}\bar{J}_n/\sigma\leq t\}-\Pr\{Z\leq t\}|\leq&
C\left(n^{-\delta(\frac{1}{2}\wedge\alpha)}\sqrt{\log n}+n^{-3\delta/16}\right)\\
\leq&C(n^{-3\delta/16}\vee n^{-\alpha\delta}\sqrt{\log n}).
\end{aligned}
\end{equation}
Taking $C$ sufficiently large so that $\eqref{eq: BE thm}$ is true
for all $n$.

\end{proof}

\renewcommand{\theequation}{A.\arabic{equation}}
\renewcommand{\thetheorem}{A.\arabic{theorem}}
\setcounter{equation}{0} \setcounter{theorem}{0}
\section*{Appendix}  

\begin{proof}[Proof of Lemma $\ref{Lemma 2.4}$]
By the definition of $C_n(s,t)$,
\begin{equation}
\begin{split}
2^{j_n}\int_{[-M,M]^2}C_n^2(s,t)dsdt&=2^{3j_n}\int_{[-M,M]^2}\left\{\int_{\rr^2}K(2^{j_n}t,2^{j_n}x)
K(2^{j_n}s,2^{j_n}x)\right.\\
&\left.\phantom{\sum
}K(2^{j_n}t,2^{j_n}y)K(2^{j_n}s,2^{j_n}y)f(x)f(y)dxdy\right\}dsdt
\end{split}
\end{equation}
By change of variables $y=x-2^{-j_n}u, t=2^{-j_n}w+x,s=2^{-j_n}z+x$
and the compactness of $\Phi$, this integral is equal to
\begin{align}
&\int_{-A}^A\int_{-A}^A\int_{-2A}^{2A}\int_\rr
K(2^{j_n}x+z,2^{j_n}x)K(2^{j_n}x+w,2^{j_n}x)
K(2^{j_n}x+z,2^{j_n}x-u)K(2^{j_n}x+w,\notag\\
&\phantom{\int}2^{j_n}x-u)f(x)f(x-2^{-j_n}u)1(2^{-j_n}z+x\in [-M,M])1(2^{-j_n}w+x\in [-M,M])dxdudzdw\notag\\
=&\int_{-A}^A\int_{-A}^A\int_{-2A}^{2A}\sum_{i=-\infty}^\infty\int_0^{2^{-j_n}}
K(2^{j_n}x+z+i,2^{j_n}x+i)K(2^{j_n}x+w+i,2^{j_n}x+i)\notag\\
&K(2^{j_n}x+z+i, 2^{j_n}x-u+i)K(2^{j_n}x+w+i,2^{j_n}x-u+i)f(x+2^{-j_n}i)f(x+2^{-j_n}i-2^{-j_n}u)\notag\\
&1(2^{-j_n}z+x+2^{-j_n}i\in [-M,M])1(2^{-j_n}w+x+2^{-j_n}i\in
[-M,M])dxdudzdw.
\end{align}
Using $K(x+1,y+1)=K(x,y)$ and change of variables, it is in turn
equal to
\begin{align}  \label{eq: 4 int}
&\int_{-A}^A\int_{-A}^A\int_{-2A}^{2A}\sum_{i=-\infty}^\infty\int_0^1
2^{-j_n}K(x+z,x)K(x+w,x)K(x+z,x-u)K(x+w,x-u) \notag \\
&f(2^{-j_n}(x+i))f(2^{-j_n}(x+i-u))1(2^{-j_n}(z+x+i)\in
[-M,M]) \notag \\
&1(2^{-j_n}(w+x+i)\in [-M,M])dxdudzdw.
\end{align}
To continue, it is convenient to write
\begin{align}  \label{eq: Ijn def}
&\quad\sum_{i=-\infty}^\infty2^{-j_n}
f(2^{-j_n}(x+i))f(2^{-j_n}(x+i-u))1(2^{-j_n}(z+x+i)\in
[-M,M])   \notag\\
&\quad1(2^{-j_n}(w+x+i)\in [-M,M])      \\
&=\left\{\sum_{i=2A}^\infty+\sum_{i=-2A}^{2A-1}+
\sum_{-\infty}^{-2A-1}\right\}2^{-j_n}f(2^{-j_n}(x+i))f(2^{-j_n}(x+i-u))  \notag\\
&\ \ \ \ 1(2^{-j_n}(z+x+i)\in [-M,M])1(2^{-j_n}(w+x+i)\in [-M,M])  \notag\\
&=:I_1(j_n)+I_2(j_n)+I_3(j_n)=I(j_n).  \notag
\end{align}
The next lemma proves the convergence of $I(j_n)$.
\begin{lemma}   \label{new lemma}
Assume that $f$ is bounded. For fixed $M>0$,
\begin{equation}
I(j_n)\to\int_{-M}^M f^2(y)dy    \label{eq: Ijn}
\end{equation}
uniformly for $x\in [0,1],u\in [-2A,2A],z\in [-A,A],w\in [-A,A]$ as
$n\to\infty$.
\end{lemma}
\begin{proof}
To simplify the notation, let $u'=x-u$, $z'=x+z$, $w'=x+w$. Then
$u'\in[-2A,2A+1]$, $z'\in[-A,A+1]$, $w'\in[-A,A+1]$. Consider
$I_1(j_n)$. The general summand of $I_1(j_n)$ is zero if
$2^{-j_n}(-A+i)>M$.
\begin{align}
I_1(j_n)&=\left(\sum_{i=2A}^{\lfloor2^{j_n}M\rfloor-2A-1}+\sum_{\lfloor2^{j_n}M\rfloor-2A}
^{\lfloor2^{j_n}M\rfloor+A}\right)2^{-j_n}
f(2^{-j_n}(x+i))f(2^{-j_n}(u'+i))\notag\\
&\quad1(2^{-j_n}(z'+i)\in [0,M])1(2^{-j_n}(w'+i)\in [0,M])\notag\\
&=: I_4(j_n)+I_5(j_n),\notag
\end{align}
where $\lfloor2^{j_n}M\rfloor$ is the largest integer less than or
equal to $2^{j_n}M$.

$I_5(j_n)$ is a finite sum with each summand bounded by a constant
times $2^{-j_n}$. So $I_5(j_n)\to0$ uniformly for $x\in [0,1],u\in
[-2A,2A],z\in [-A,A],w\in [-A,A]$.

Setting $\triangle y=2^{-j_n}(4A+1)$, we can simplify $I_4(j_n)$
since the indicator function in the general summand of $I_4(j_n)$
must be 1.
\begin{equation}
\begin{split}   \label{eq: I4jn}
I_4(j_n)&=\sum_{i=2A}^{\lfloor2^{j_n}M\rfloor-2A-1}2^{-j_n}
f(2^{-j_n}(x+i))f(2^{-j_n}(u'+i))\\
&=\frac{1}{4A+1}\sum_{i=2A}^{6A}\sum_{j=0}^{N_i}\triangle y
f(2^{-j_n}(x+i)+j\triangle y)f(2^{-j_n}(u'+i)+j\triangle y),
\end{split}
\end{equation}
where $N_i$ is the largest $j$ such that for fixed $i$,
$i+j(4A+1)\leq\lfloor2^{j_n}M\rfloor-2A-1$. $N_i=\lfloor M/\triangle
y-1\rfloor$ or $\lfloor M/\triangle y-2\rfloor$ depending on $i$.

For each $2A\leq i\leq6A$, consider the partition of $[0,M]$:
$$P_{i,n}=\{0, 2^{-j_n}(i-2A), 2^{-j_n}(i-2A)+\triangle y,..., 2^{-j_n}(i-2A)+(N_i+1)\triangle y,
M\}.$$ There are at most $N_i+3$ subintervals. Except for the first
and the last subintervals, whose lengths we denote respectively by
$\triangle y_{i,1}$ and $\triangle y_{i,N_i+3}$, all the
subintervals in this partition have length $\triangle
y=2^{-j_n}(4A+1)$. We also have $0\leq\triangle y_{i,1}\leq\triangle
y$ and $0\leq\triangle y_{i,N_i+3}\leq\triangle y$. Setting
\begin{equation}   \label{eq: Sni}
S_{i,n}:=f^2(0)\triangle y_{i,1}+\sum_{j=0}^{N_i}\triangle y
f(2^{-j_n}(x+i)+j\triangle y)f(2^{-j_n}(u'+i)+j\triangle
y)+f^2(M)\triangle y_{i,N_i+3},
\end{equation}
we see that
\begin{equation}   \label{eq: Sni lb}
S_{i,n}\leq f^2(0)\triangle y_{i,1}+\sum_{j=0}^{N_i}
M_{i,j}^2\triangle y+f^2(M)\triangle y_{i,N_i+3}
\end{equation}
and
\begin{equation}  \label{eq: Sni ub}
S_{i,n}\geq f^2(0)\triangle y_{i,1}+\sum_{j=0}^{N_i}
m_{i,j}^2\triangle y+f^2(M)\triangle y_{i,N_i+3},
\end{equation}
where $M_{i,j}$ and $m_{i,j}$ denote respectively the supremum and
the infimum of $f$ on the partition $[2^{-j_n}(i-2A)+j\triangle y,
2^{-j_n}(i-2A)+(j+1)\triangle y]$. As $n\to\infty$, the mesh of
$P_{i,n}$ tends to zero. Obviously, $f^2(0)\triangle
y_{i,1}+f^2(M)\triangle y_{i,N_i+3}\to 0$. $f\in L_1$ and boundness
of $f$ implies that $f^2$ is Riemann integrable on $[0,M]$ for any
$M>0$. It follows that $S_{i,n}\to\int_0^Mf^2(y)dy$ for $2A\leq
i\leq 6A$ and by $\eqref{eq: I4jn}$,
\begin{equation}
I_4(j_n)\to\int_0^Mf^2(y)dy.
\end{equation}
Note that this convergence is uniform for $x\in [0,1]$ and $u'\in
[-2A,2A+1]$, therefore, it is uniform for $x\in [0,1],u\in
[-2A,2A],z\in [-A,A],w\in [-A,A]$. We have thus proved that
$\lim_{n\to\infty}I_1(j_n)=\int_0^Mf^2(y)dy$ uniformly for $x, u, z,
w$ in the corresponding intervals. By analogy,
$I_3(j_n)\to\int_{-M}^0f^2(y)dy$ uniformly for $x, u, z, w$ in the
same intervals.

Since $f$ is bounded, $I_2(j_n)\to0 \ \ {\rm as} \ \ n\to\infty.$
$\eqref{eq: Ijn}$ is proved when collecting the results for
$I_1(j_n)$, $I_2(j_n)$ and $I_3(j_n)$.
\end{proof}

\begin{lemma}     \label{Lemma K int}
Assume the scaling function $\phi$ satisfies (S1) such that the
kernel K associated with $\phi$ is dominated by $\Phi$ whose support
is contained in $[-A,A]$, where $A$ is an integer. Then
\begin{equation}
\int_{-A}^A\int_{-A}^A\int_{-2A}^{2A}\int_0^1K(x+z,x)K(x+w,x)K(x+z,x-u)K(x+w,x-u)dxdudzdw
=1.
\end{equation}
\end{lemma}
\begin{proof}
Since $K(x+z,x)K(x+w,x)K(x+z,x-u)K(x+w,x-u)$ is absolutely
integrable, by Fubini's theorem,
\begin{equation}
\begin{split}
&\quad\int_{-A}^A\int_{-A}^A\int_{-2A}^{2A}\int_0^1K(x+z,x)K(x+w,x)K(x+z,x-u)K(x+w,x-u)dxdudzdw\\
&=\int_0^1\int_\rr\int_\rr K(x+z,x)K(x+z,x-u)dz\int_\rr
K(x+w,x)K(x+w,x-u)dwdudx.
\end{split}
\end{equation}
We make the following observation: For any $y$ and $z$, by
orthogonality of $\phi$,
\begin{align}   \label{eq: intK}
&\int K(x,y)K(x,z)dx    \notag\\
=&\int\sum_k\phi^2(x-k)\phi(y-k)\phi(z-k)dx+\int\sum_{k\neq
l}\phi(x-k)\phi(y-k)\phi(x-l)\phi(z-l)dx  \notag\\
=&\sum_{k\in\zz}\phi(y-k)\phi(z-k)\int\phi^2(x-k)dx+\sum_{k\neq
l}\phi(y-k)\phi(z-l)\int\phi(x-k)\phi(x-l)dx \notag\\
=&\sum_{k\in\zz}\phi(y-k)\phi(z-k)=K(y,z).
\end{align}
For fixed $x\in[0,1]$, by repeated applications of the above
equation,
\begin{equation}
\begin{split}
&\quad\int_\rr\int_\rr K(x+z,x)K(x+z,x-u)dz\int_\rr
K(x+w,x)K(x+w,x-u)dwdu\\
&=K(x,x)=\sum_{k\in\zz}\phi^2(x-k).
\end{split}
\end{equation}
Finally we consider
\begin{equation}
\int_0^1\sum_{k\in\zz}\phi^2(x-k)dx
=\sum_{k\in\zz}\int_0^1\phi^2(x-k)dx=\int\phi^2(x)dx=1.
\end{equation}
\end{proof}

We now continue with the proof of Lemma $\ref{Lemma 2.4}$. Since in
Lemma $\ref{new lemma}$, the convergence is uniform for $x\in
[0,1],u\in [-2A,2A],z\in [-A,A],w\in [-A,A]$, then if $n$ is
sufficiently large, for fixed $M>0$,
\begin{equation}      \label{eq: Ijn bd}
|I(j_n)|\leq2\int_{-M}^Mf^2(t)dt.
\end{equation}
The quantity in $\eqref{eq: 4 int}$ is bounded in absolute value by
\begin{equation}
\begin{split}           \label{eq: fubini}
&\quad\|\Phi\|_\infty^4\int_{-A}^A\int_{-A}^A\int_{-2A}^{2A}\int_0^1I(j_n)dxdudzdw\\
&\leq2\|\Phi\|_\infty^4\int_{-A}^A\int_{-A}^A\int_{-2A}^{2A}\int_0^1\int_{-M}^Mf^2(t)dtdxdudzdw
<\infty
\end{split}
\end{equation}
for $n$ large. So, by Fubini, $\eqref{eq: 4 int}$ is equal to
\begin{equation}
\int_{-A}^A\int_{-A}^A\int_{-2A}^{2A}\int_0^1K(x+z,x)K(x+w,x)K(x+z,x-u)K(x+w,x-u)I(j_n)dxdudzdw.
\end{equation}
By dominated convergence and Lemmas $\ref{new lemma}$, $\ref{Lemma K
int}$, it converges to $\int_{-M}^M f^2(y)dy$.
\end{proof}

\begin{proof}[Proof of Lemma $\ref{Lemma Cn2}$]
Choosing $M$ to be an integer such that $M\geq L+2^{-j_n}(4A+1)$, we
divide the plane $\rr^2$ into four regions: $[-M,M]^2,
[-M,M]^C\times[-M,M]^C, [-M,M]\times[-M,M]^C$ and
$[-M,M]^C\times[-M,M]$. To get the rate at which
$2^{j_n}\int_{[-M,M]^2}C_n^2(s,t)dsdt$ tends to $\int_{-M}^M
f^2(y)dy$, we estimate $\left|I(j_n)-\int_{-M}^M f^2(y)dy\right|$.

$I_1(j_n)$, which was defined in $\eqref{eq: Ijn def}$, can be
decomposed into 4 terms as follows.
\begin{equation}
\begin{aligned}
I_1(j_n)&=\left(\sum_{i=2A}^{\lfloor2^{j_n}L\rfloor-2A-1}+\sum_{\lfloor2^{j_n}L\rfloor-2A}
^{\lceil2^{j_n}L\rceil+2A-1}+\sum_{i=\lceil2^{j_n}L\rceil+2A}^{2^{j_n}M-2A-1}
+\sum_{i=2^{j_n}M-2A}^{2^{j_n}M+A}\right)2^{-j_n}f(2^{-j_n}(x+i))\\
&\quad f(2^{-j_n}(u'+i))1(2^{-j_n}(z'+i)\in [0,M])1(2^{-j_n}(w'+i)\in [0,M])\\
&=: I_4'(j_n)+I_5'(j_n)+I_6'(j_n)+I_7'(j_n).
\end{aligned}
\end{equation}
$I_4'(j_n)$ is essentially the same as $I_4(j_n)$ in $\eqref{eq:
I4jn}$. We follow the argument from $\eqref{eq: I4jn}$ to
$\eqref{eq: Sni ub}$ but consider the interval $[0,L]$ instead. Due
to the hypothesis of H\"{o}lder continuity, there exists $C$
depending on $f$ and $\{j_n\}$, such that
\begin{equation}
|M_{ij}^2-m_{ij}^2|\leq|M_{ij}+m_{ij}||M_{ij}-m_{ij}|\leq
C(\triangle y)^\alpha\leq Cn^{-\delta\alpha}.
\end{equation}
So we obtain
\begin{equation}
\left|S_{i,n}-\int_{0}^Lf^2(y)dy\right|\leq CLn^{-\delta\alpha}.
\end{equation}
Obviously, $f^2(0)\triangle y_{i,1}$ and $f^2(L)\triangle
y_{i,N_i+3}$ are both bounded by $Cn^{-\delta}$. From $\eqref{eq:
I4jn}$, for all $x\in[0,1]$, $u'\in[-2A,2A+1]$, $z'\in[-A,A+1]$,
$w'\in[-A,A+1]$,
\begin{equation}
\begin{aligned}
\left|I_4'(j_n)-\int_{0}^Lf^2(y)dy\right|
&\leq\frac{1}{4A+1}\sum_{i=2A}^{6A}\left|S_{i,n}-f^2(0)\triangle
y_{i,1}-f^2(L)\triangle y_{i,N_i+3}-\int_{0}^Lf^2(y)dy\right|\\
&\leq Cn^{-\delta\alpha}
\end{aligned}
\end{equation}
for $n$ large enough depending on $\{j_n\}$. $C$ depends on $f$ and
$\{j_n\}$.

Next we will look at $I_6'(j_n)$ and consider a partition $P_{i,n}$
on $[L,M]$. Let $\xi_{ij}:=2^{-j_n}(x+i)+j\triangle y$,
$\xi_{ij}'=2^{-j_n}(u'+i)+j\triangle y$. Similar to $\eqref{eq:
I4jn}$, but for a different $N_i$, we write,
\begin{equation}
\begin{aligned}
I_6'(j_n)&=\frac{1}{4A+1}\sum_{i=\lceil2^{j_n}L\rceil+2A}^{\lceil2^{j_n}L\rceil+6A}\sum_{j=0}^{N_i}\triangle
yf(\xi_{ij})f(\xi_{ij}').
\end{aligned}
\end{equation}
Since $f$ is bounded and monotonically decreasing on $[L,\infty)$,
it follows that
\begin{equation}
\int_{L+\triangle y_{i,1}+\triangle y}^{M-\triangle
y_{i,N_i+3}}f^2(y)dy\leq\sum_{j=0}^{N_i}\triangle
yf(\xi_{ij})f(\xi_{ij}')\leq C\triangle y+\int_{L+\triangle
y_{i,1}}^{M-\triangle y_{i,N_i+3}-\triangle y}f^2(y)dy.
\end{equation}
Thus when $M\geq L+2^{-j_n}(4A+1)$, for all $x\in[0,1]$,
$u'\in[-2A,2A+1]$, $z'\in[-A,A+1]$, $w'\in[-A,A+1]$ and $n$ large
enough depending on $\{j_n\}$,
\begin{equation}
\left|I_6'(j_n)-\int_L^Mf^2(y)dy\right|\leq Cn^{-\delta},
\end{equation}
where $C$ depends on $f$ and $\{j_n\}$. We also have
$|I_5'(j_n)|\leq Cn^{-\delta}$ and $|I_7'(j_n)|\leq Cn^{-\delta}$.
Collecting these bounds,
\begin{equation}
\left|I_1(j_n)-\int_0^M f^2(y)dy\right|\leq
C(n^{-\delta\alpha}+n^{-\delta})\leq Cn^{-\delta\alpha}.
\end{equation}
Now it's easy to see $\left|I(j_n)-\int_{-M}^M f^2(y)dy\right|\leq
Cn^{-\delta\alpha}$. By $\eqref{eq: 4 int}$, $\eqref{eq: Ijn def}$
and Lemma $\ref{Lemma K int}$, we get
\begin{equation}   \label{eq: M2}
\left|2^{j_n}\int_{[-M,M]^2}C_n^2(s,t)dsdt-\int_{-M}^M
f^2(y)dy\right|\leq Cn^{-\delta\alpha}.
\end{equation}
The derivation of a bound on
$\left|2^{j_n}\int_{[-M,M]^C}\int_{[-M,M]^C}C_n^2(s,t)dsdt-\int_{[-M,M]^C}f^2(y)dy\right|$
is similar. The analysis of the key component is analogous to
$I_6'(j_n)$, where the monotonicity of the tail of $f$ is used.
\begin{equation}
\left|2^{j_n}\int_{[-M,M]^C}\int_{[-M,M]^C}C_n^2(s,t)dsdt-\int_{[-M,M]^C}f^2(y)dy\right|\leq
Cn^{-\delta}.
\end{equation}
It's easier to analyze the integral on the regions
$[-M,M]\times[-M,M]^C$ and $[-M,M]^C\times[-M,M]$. Both are bounded
by $Cn^{-\delta}$ since there are at most finitely many summands
that are not zero.  $\eqref{eq: Cn2 bd}$ follows by collecting the
bounds on the four regions and taking $C$ sufficiently large so that
it is true for all $n$.

\end{proof}

\section*{Acknowledgement}
I would like to express my sincere gratitude to my advisor Prof.
Evarist Gin\'{e} for his constant support during the dissertation. I
appreciate his patience, numerous hours of insightful discussions
and careful proofreading of the manuscript. It would have been
almost impossible for me to write this article without his help.

\bibliographystyle{model2-names}

\begin{thebibliography}{[50]}

  \bibitem {Bowman 1985}
    Bowman, A.W. (1985)
    \emph{A comparative study of some kernel-based nonparametric density
    estimators}. J. Stat. Comput. Simul. {\bf21}, 313-327.

  \bibitem{dela Pena 1999}
     de la Pe\~{n}a, V. and Gin\'{e} E. (1999)
     \emph{Decoupling: From Dependence to Independence},
     Springer, New York.

   \bibitem{Dou 1990}
     Doukhan, P. and Le\'{o}n, J.R. (1990)
     \emph{D\'{e}viation quadratique d'estimateurs de densit\'{e} par projections orthogonales}.
     C.R. Acad. Sci. Paris Sr. I {\bf 310}, 425-430.

  \bibitem{Dou 1993}
     Doukhan, P. and Le\'{o}n, J.R. (1993)
     \emph{Quadratic deviation of projection density estimates}.
     Rebrape {\bf 7}, 37-63.

  \bibitem {DKW 1956}
     Dvoretzky, A., Kiefer, J. and Wolfowitz, J. (1956)
     \emph{Asymptotic minimax character of the sample distribution function and of the classical multinomial
     estimator}.
     Ann. Math. Statist. {\bf27}, 642-669.

  \bibitem {EQW 1979}
     Erickson, R.V., Quine M.P. and Weber N.C. (1979)
     \emph{Explicit bounds for the departure from normality of sums of dependent random variables}.
     Acta Math. Acad. Sci. H. {\bf34}, 27-32.

  \bibitem{Folland 1999}
     Folland G.B. (1999)
     \emph{Real Analysis: Modern Techniques and Their Applications}, 2nd
     ed., Wiley, New York.

  \bibitem{Gine 2000}
     Gin\'{e}, E., Lata{\l}a, R. and Zinn, J. (2000)
     \emph{Exponential and moment inequalities for U-statistics}.
     In: Gin\'{e}, E., Mason, D.M., Wellner, J.A. (eds.): High
     dimensional probability II, Progr. Probab. {\bf47}, 13-38.

  \bibitem{Gine 2004}
     Gin\'{e}, E. and Mason D.M. (2004)
     \emph{The law of the iterated logarithm for the integrated squared deviation of a kernel density
     estimator}.
     Bernoulli, {\bf10}, 721-752.

   \bibitem{Hall 1984}
     Hall, P. (1984)
     \emph{Central limit theorem for integrated square error of multivariate nonparametric density
     estimators}.
     J. Multivariate Anal., {\bf14}, 1-16.

   \bibitem {HKPT 1998}
     H\"{a}rdle, W., Kerkyacharian, G., Picard, D. and Tsybakov, A. (1998)
     \emph{Wavelets, approximation, and statistical applications}.
     Lecture Notes in Statistics {\bf 129}, Springer, New York.

  \bibitem {KP 1992}
     Kerkyacharian, G. and Picard, D. (1992)
     \emph{Density estimation in Besov spaces}.
     Statist. Probab. Lett. {\bf 13}, 15-24.

  \bibitem{Komlos 1975}
     Koml\'{o}s, J., Major, P. and Tusn\'{a}dy, G. (1975)
     \emph{An approximation of partial sums of independent rv's and the sample distribution function},
     I. Z. Wahrscheinlichkeitstheorie Verw. Geb., {\bf32}, 111-131.

  \bibitem{Montgomery 1993}
     Montgomery-Smith, S.J. (1993)
     \emph{Comparison of sums of independent identically distributed random vectors},
     Probab. Math. Statist. {\bf14}, 281-285.

  \bibitem{Pinsky 1966}
     Pinsky, M. (1966)
     \emph{An elementary derivation of Khintchine's estimate for large deviations},
     Probab. Math. Statist. {\bf14}, 281-285.

  \bibitem{Shorack 1986}
     Shorack, G. and Wellner, J. (1986)
     \emph{Empirical Processes with Applications to Statistics},
     Wiley, New York.

  \bibitem{Zhang 1999}
     Zhang, S. and Zheng Z. (1999)
     \emph{On the asymptotic normality for $L_2$-error of wavelet density estimator with application},
     Comm. Statist. Theory Methods. {\bf28}, 1093-1104.

 \end{thebibliography}

\end{document}